\documentclass[epsfig,leqno,12pt]{amsart}

 \usepackage{eepic,epic}            
\usepackage{graphicx}
\renewcommand{\theequation}     
    {\arabic{section}.\arabic{subsection}.\arabic{equation}}

\theoremstyle{plain}
\newtheorem*{theorem}{Theorem}

\newtheorem*{corollary}{Corollary}

\theoremstyle{remark}

\newtheorem*{remark}{Remark}

\newtheorem*{definition}{Definition}

\def\U{{\frak U}}

\def\AA{{{\mathrm{A}}_{\frac{1}{2}\infty}}}
\def\DD{{{\mathrm{D}}_{\frac{1}{2}\infty}}}
\def\Z{{\mathbf Z}}

\def\C{{\mathbf C}}
\def\R{{\mathbf R}}
\def\P{{\mathbf P}}

\def\Id{{\operatorname{Id}}}

\def\Hom{{\operatorname{Hom}}}

\def\ker{{\operatorname{ker}}}

\def\id{{\operatorname{id}}}
\def\rank{{\operatorname{rank}}}

\begin{document}
\title[Coxeter elements for  vanishing\\ cycles of types $\AA$ and $\DD$]
{Coxeter elements for vanishing \\ cycles of types $\AA$ and $\DD$}%
\author{Kyoji Saito}
\address{ IPMU, university of Tokyo}

 \begin{abstract}
We introduce two real entire functions $f_{\AA}$and $f_{\DD}$ in two
  variables. Both of them have only two critical values $0$ and $1$, and
  the associated maps $\C^2\!\to\! \C$ define topologically locally trivial fibrations over
  $\C\!\setminus\!\{0,1\}$. All critical points are
  ordinary double points, and the associated vanishing cycles span
  the middle homology group of the general fiber, whose intersection
  diagram forms bi-partitely decomposed quivers of type $\AA$ and $\DD$, respectively.
Coxeter elements of type $\AA$ and $\DD$, acting on the
  middle homology group, are introduced as the product of
  the monodromies around $0$ and $1$.
We describe the spectra of the Coxeter
  elements by embedding the middle homology group into a Hilbert
  space. The spectra turn out to be strongly continuous on the
  interval $(-\frac{1}{2},\frac{1}{2})$ except at 0 for type $\DD$.
\end{abstract} \footnote
{Present paper is planned as the first part of a paper ``Primitive forms
of types $\AA$ and $\DD$'' in preparation. We publish the present part
(the spectra of Coxeter elements) separately, because of its own
independent interest.
}

\vspace{-0.4cm}

\maketitle   
\vspace{-0.5cm}
\tableofcontents
\vspace{-0.7cm}

\newpage
\section{\vspace{-0.1cm}
Introduction}

\vspace{-0.1cm}
We introduce two particular entire transcendental functions in two variables,
which we will call the functions of types $\AA$ and $\DD$,
respectively. 
They are introduced in the hope that some period maps
associated with them should contribute
to the understanding of KP- and KdV-hierarchies. For this purpose, we need to
develop a theory of primitive forms for these transcendental functions
by analogy with the classical theory of primitive forms for polynomial
local singularities {\small \cite{Sa1}},\footnote
{In the classical
theory, a primitive form is defined on the universal
unfolding of a function having an isolated critical point. Whereas, in
the present program, the ``generating center of the unfolding''
is these functions of types {\tiny $\AA$ and $\DD$,} where there is
not yet a general frame work available.} 
where the data of spectrum of Coxeter elements is basic. Therefore, as a first
step, in the
present paper, we study  the spectrum of the 
large circle monodromy, called the Coxeter element, acting on the lattice of vanishing\! cycles of these
functions.\ The goal is to show that the spectrum is contained in the
interval 
$(\!$-$\!\frac{1}{2},\!\frac{1}{2}\!)$. 

We are still in early stage in studing transcendental
functions in such geometric contexts.
Many of classical algebraic tools are not available due to the lack of
compactness/finiteness nature of them. However, the
transcendency of the functions which we study in the present paper as
the test cases, is
still not ``wild'', and we handle them by ``hand''. Even\! though each\!
step\! of\! the calculation is elementary,\! we\ want to be cautious and will
proceed with the calculations in down to earth fashion.

Let us explain the contents of the paper in more details.  We first
explain a classical analogue of the present work {\small (\cite[\S2.5,3]{Sa2})}, and then 
make some comparisons between the classical case and the present case. 

For a Dynkin graph $\Gamma_W$ of type {\small $W\!\in\!\!\{A_l \ (l\!\in\!\!\Z_{\ge\!1}\!), D_l\
(l\!\in\!\!\Z_{\ge\!4}\!), E_6, E_7, E_8\}$}, there exists a real polynomial $f_W(x,y,z)$ with the following
i)-iii).\footnote
{This is a consequence of a result in \cite[\S2.5,\! 3]{Sa2}(\cite[\S6.5
Remark 19 and \S8.9 Remark 27]{Sa3}). Let us briefly recall the
result. For each simply-laced Dynkin type $W$,
there exists two parameter family $f_W(x,y,z,R,S)$ of real polynomials
of type $W$, having only two critical values and having properties ii)
and iii) (choose $f_W$ such that the its deformation 
class $[f_W]$ belongs to the vertex orbit line $O$ in the deformation
parameter space $T_W$ of simple polynomials of type $W$). Then, fix the
parameter values of $(R,S)$ to re-size the critical values
to $\{0,1\}$. In particular, if $W\!=\!A_l$, the polynomial is given by
(Chebyshev polynomial in $x$) $+y^2+z^2$.

 Instead of the formulation in 3-variables as given here, we may
 formulate results in 2-variables by replacing an intersection
diagram by a quiver diagram (see \S3.3).
%
}

i) {\it All critical points of $f_W$ are Morse (i.e.\ 
{\small Hessians at the critical points are non-degenerate}), and $f_W$ has only two critical values $0$ and $1$. }

ii) {\it The map $f_W:\C^3\to \C$ is a locally trivial
fibration over $\C\setminus\{0,1\}$. Let us denote by $X_t$ the fiber
$f_W^{-1}(t)$ over a point $t\in\C$.}

iii) {\it For $t\in(0,1)$, let
$\{\gamma_{0}^{(i)}\}_{i\in C_0}$ (resp.\ $\{\gamma_{1}^{(i)}\}_{i\in C_1}$) be
the set of cycle in the middle homology group $\mathrm{H_2}(X_t\Z)$
which vanish at a critical point in the fiber $X_0$ as $t\!\downarrow\!0$ (resp.\
$X_1$ as $t\!\uparrow\!1$). Then, a) the union $\{\gamma_{0}^{(i)}\}_{i\in
C_0}\cup\{\gamma_{1}^{(i)}\}_{i\in C_1}$ forms a basis of
$\mathrm{H_2}(X_t,\Z)$, b) the intersection diagram of the basis
gives a bi-partite decomposition of the Dynkin graph $\Gamma_W$.}

In the first half of the present paper, we show that the functions of type 
$\AA$ and $\DD$ satisfy exactly the properties parallel to i), ii) and
iii) by replacing $\Gamma_W$ by the infinite quivers of type $\AA$
and $\DD$ introduced in \S3.2. This fact explains the naming of the
functions. Here we should note that the middle homology group is of
infinite rank.

In the classical polynomial $f_W$ case, the product of the two monodromies of the
fibration around $0$ and around $1$, acting on the lattice
$\mathrm{H_2}(X_t,\Z)$, is called the {\it Coxeter element}. The eigenvalues of the Coxeter elements are given by the set
{\small $\mathrm{exp}(2\pi\sqrt{-1}\frac{m_i}{h})$ ($i\!=\!1,\!\cdots\!,l$)}, where
{\small $h\in\Z_{>0}$} is the Coxeter number of type $W$ and
{\small $0\!<\!m_1\!<\!m_2\!\le\!\cdots\!<\!m_l\!<\!h$} are called exponents (see {\small \cite[ch.V,\S6,n$^o$2]{Bo}}).
The data of exponents, or equivalently, the spectrum
$\frac{m_i}{h}$ {\small ($i\!=\!1,\!\cdots\!,l$)} are quite important both for the
Lie theory of type $W$ {\small \cite{Bo}} and for the primitive forms of type $W$
{\small \cite{Sa1}}. For instance, the fact
that the spectrum is contained in the interval
$(\frac{d}{2}\!-\!1,\frac{d}{2})$ means that the 
primitive form is of ``simple type''  of dimension $d$ (see Remark at the
end of \S3 and {\small (\cite{Sa4})}. However, the eigenvalues of the Coxeter
element alone are not sufficient to recover the spectrum due to an
ambiguity modulo integers. The clue to recover the 
spectrum is the study of eigenvalues of the intersection form on
the root lattice, as will be described in \S4 in the present paper.

Returning to the transcendental functions of types $\AA$
and $\DD$, in analogy with the classical case, we introduce the Coxeter element as the product of the local
monodromy around $0$ and around $1$. In order to be able to discuss
about the ``eigenvalues'' of the Coxeter element,
we embed the ``lattice'' into the Hilbert space (see \S4.1) such
that the simple root
basis turn to the ortho-normal basis of the Hilbert space.\footnote
{This
view point is already implicitly in {\small \cite[ch.V,\S6,n$^o$2]{Bo}}. The Hilbert
space lies between the homology group 
$\mathrm{H}_1(X_t,\C)$ and the cohomology group $\mathrm{H}^1(X_t,\C)$ (i.e.\ a sort of
``intersection cohomology group'', See end of \S4.2), which fit with our
original intention to develop a period map theory for this cohomology groups.}

The main result of the present paper is that the spectrum of the
 Coxeter element is given by the function 
 $\theta\!-\!\frac{1}{2}$ on the interval $\theta\in[0,1]$ with a Stieltjes measure 
$\xi_{W,\theta}$ which is strongly continuous
 (\S4 Theorem 7). Actually, $\xi_{W,0}\!=\!\xi_{W,1}\!=\!0$ (i.e.\
 there is no discrete spectrum at $\theta=0,1$) so that the spectrum is
 contained in the open interval $(-\frac{1}{2},\frac{1}{2})$. 

This is what was expected from the analogy with classical theory.


\newpage
\section{Functions of types $\AA$ and $\DD$ } 

\noindent
We introduce functions of type $\AA$ and $\DD$ and associated fibrations.

\subsection{Definition of  {\small $f_{\AA}$} and {\small $f_{\DD}$}}
\begin{definition}
The function $f_{W}$ of type $W\in\{\AA,\DD\}$ \footnote
{In the present paper, the expression ``{\it of type} $W$" automatically
 implies $W\in \{\AA,\DD\}$. Meaning for this name is given in \S3.2 Quiver and
 its Remark.
}
 is a {\it real entire function}\footnote{
We mean by a {\it real entire function of n-variables} a holomorphic
 function on $\C^n$ which is real valued on the real form $\R^n$ of $\C^n$.
}
in two variables $x$ and $y$ given by
\begin{eqnarray}
\label{eq:1}
f_{\AA}(x,y) &\!:=\! &\!x s^2(x) -y^2 \ \ =\ 1-c^2(x)-y^2 \\
\label{eq:2}
f_{\DD}(x,y) &\!:=\! &\!x s^2(x)-xy^2 \ =\ 1-c^2(x)-xy^2 .
\end{eqnarray}
Here $s(x)$ and $c(x)$ are real entire functions \footnote
{In the sequel of the present paper, we shall freely use the following equalities:
\centerline{
$c(0)\!=\!s(0)\!=\!1, \ \ \  xs^2(x)\!+\!c^2(x)\!=\!1, \ \ \  s'(x)\!=\!\frac{1}{2x}(c(x)\!-\!s(x))\ \ \ \text{and}\ \ \ c'(x)\!=\!-\!\frac{1}{2}s(x)$
}
without referring to them explicitly (here $f'(x)\!=$the differentiation of $f(x)$).
}
in a variable $x$ given by
{\small
\begin{eqnarray}
\label{eq:3}
 s(x)&:= & \frac{\sin{\sqrt{x}}}{\sqrt{x}} \ =\  \prod_{n=1}^{\infty} \big(1-\frac{x}{n^2\pi^2}\big)  \\
\label{eq:4}
c(x)& := &\cos{\sqrt{x}} \ =\ \prod_{n=1}^\infty \big(1-\frac{4x}{(2n-1)^2\pi^2}\big).
\end{eqnarray}
}
\end{definition}

\subsection{Real level sets {\footnotesize $X_{\AA,0,\R}$} and {\footnotesize $X_{\DD,0,\R}$}}.

We introduce the real level-0 set of the function $f_W$ of type $W$ by   
\[
X_{W,0,\R}\ :=\ \R^2\cap f_W^{-1}(0)\ .
\]
Conceptual figures of them are drawn in the following.

\vspace{1.cm}

\noindent
$
\begin{array}{ll}
\text{Figure 1}\\
{}\\
X_{A_{\frac12\infty},0,\mathbb{R}}
\end{array}
$
\quad
\smash{\raisebox{-1cm}{%
\includegraphics[width=10cm,height=2.16cm]{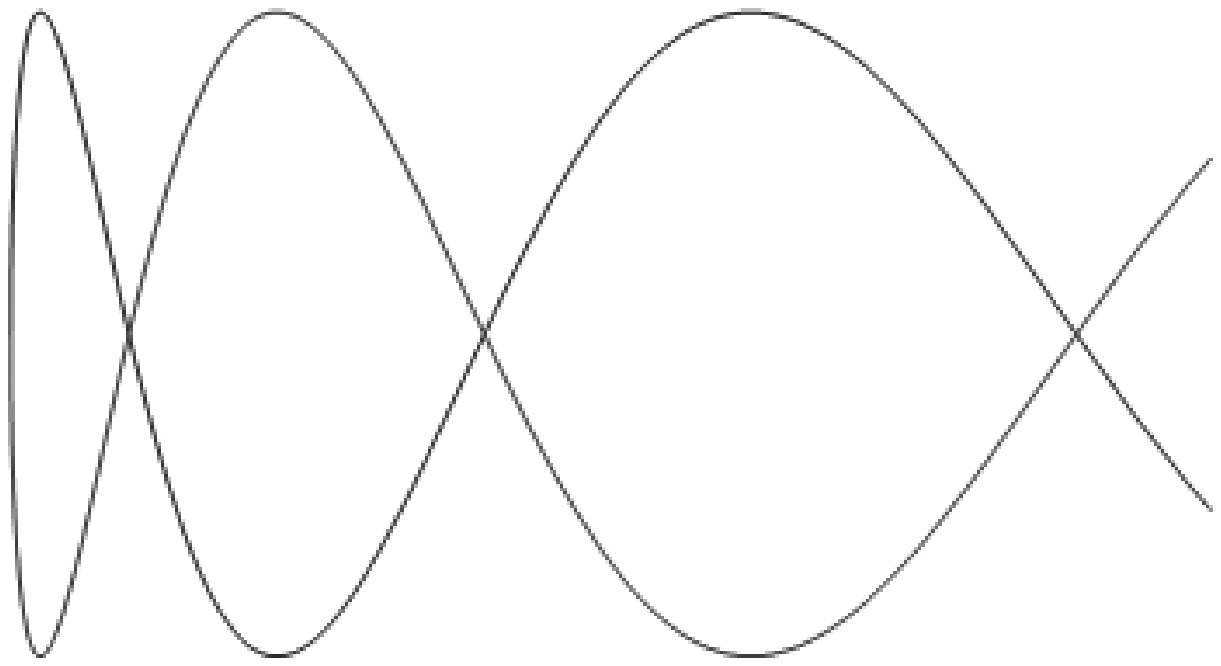}}}

\vspace{1.cm}

\noindent
$
\begin{array}{ll}
\text{Figure 2}\\
\\
X_{D_{\frac12\infty},0,\mathbb{R}}
\end{array}
$
\quad
\smash{\raisebox{-1cm}{%
\includegraphics[width=10cm,height=2.16cm]{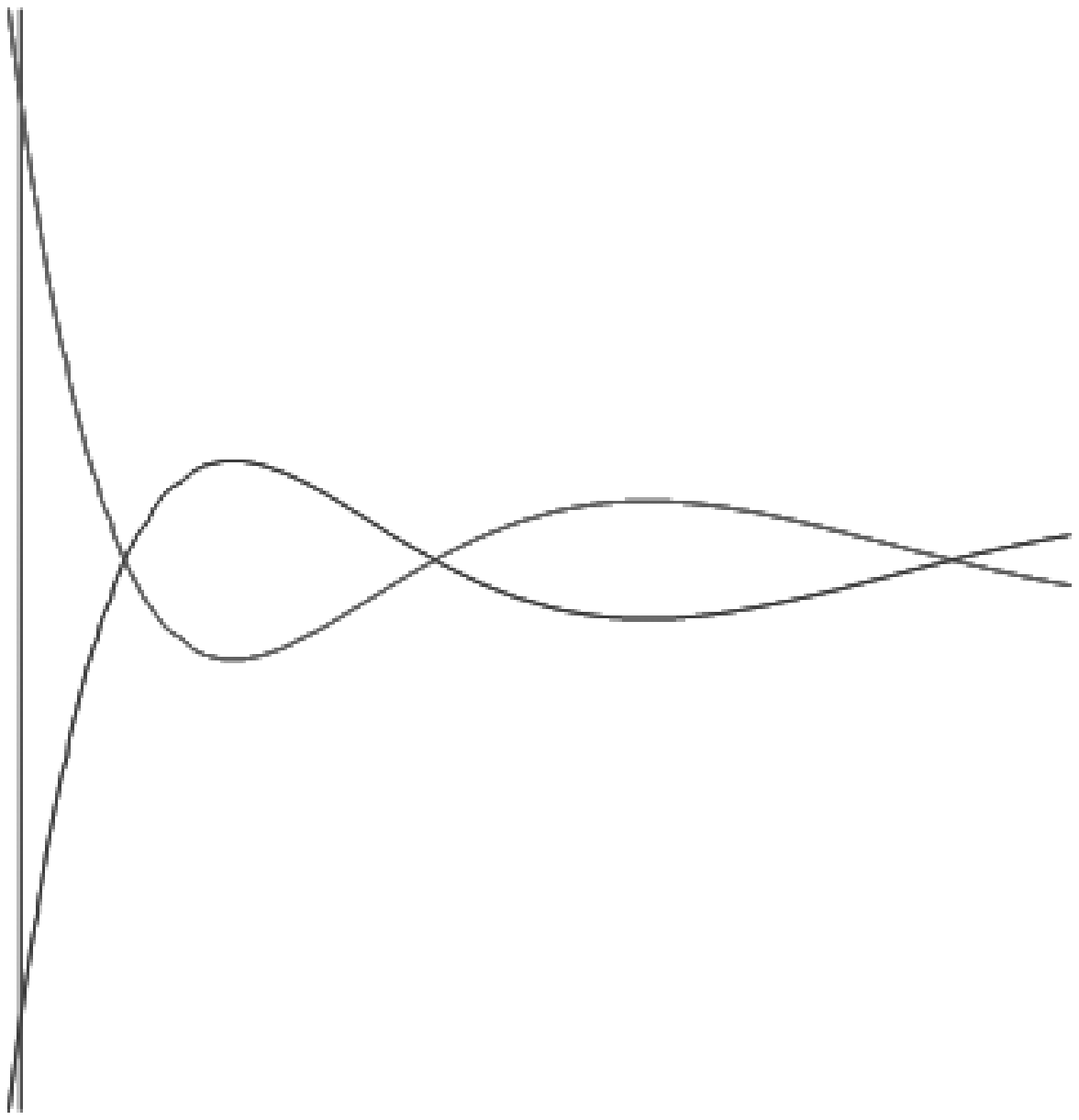}}}
\vspace{-2.0cm}

\newpage
\noindent
{\bf Terminology} 1. By a {\it bounded connected component} ({\it bcc} for short) of type $W$, we mean a bounded connected component of $\R^2 \setminus X_{W,0,\R}$. 

2. By a {\it node} of type $W$, we  mean a point on the real curve $X_{W,0,\R}$ where two local smooth irreducible components are crossing normally.

3. We say that a node of type $W$ is {\it adjacent} to a bcc of type $W$ if the node belongs to the closure of the bcc.

\medskip
We state some immediate observations on the level set $X_{W,0,\R}$, which can be
easily verified by a use of absolutely convergent infinite products
\eqref{eq:3} and \eqref{eq:4}. 

\medskip
\noindent
{\bf Observation 1.}\ {\it For $n\!=\!0,1,2,\cdots$, there exists exactly one bounded connected component of type $W$, containing the interval $(n^2\pi^2,(n\!+\!1)^2\pi^2)$ on the $x$-axis and contained in the domain $(n^2\pi^2,(n\!+\!1)^2\pi^2)\times y$-axis.}

{\bf 2.} {\it For $n=1,2,3,\cdots$, the point
$c_{W,0}^{(n)}:=(n^2\pi^2,0)$ on the $x$-axis is a node of type $W$,
which is adjacent to two bcc containing the interval $((n\!-\!1)^2\pi^2,n^2\pi^2)$ and the interval $(n^2\pi^2,(n\!+\!1)^2\pi^2)$. }

\subsection{Fibrations over {\footnotesize $\C\setminus\{0,1\}$}}.

For each type $W\in \{\AA,\DD\}$, let us consider a holomorphic map 
\begin{equation}
\label{eq:5}
f_W\ :\ \mathbf{X}_W\ \longrightarrow\ \C,
\end{equation}
where the domain $\mathbf{X}_W\!:=\!\C^2$ of $f_W$ is regarded as a contractible Stein manifold equipped with the real form $\R^2$.
The fiber $X_{W,t}:=f_W^{-1}(t)$ over $t\in \C$ is an {\it open} Riemann surface, closely embedded in $\C^2$. 

\smallskip 
\begin{remark} As we shall see in sequel, the fiber $X_{W,t}$ ($t\in\C$)
 has infinite genus. It is ``wild'' in the sense that the closure $\bar X_{W,t}$ in $\P^2_\C$ is equal to $X_{W,t}\cup \P^1_\C$ (i.e.\ the ``ends'' of $X_{W,t}$ is the $\P^1_\C$, this fact can be easily shown by the value distribution theory of one variable). 
By putting
\begin{equation}
\label{eq:6}
\bar{\mathbf{X}}_W\ :=\ \mathbf{X}_W \cup (\P^1_\C\times \C) \ := \cup_{t\in\C} (\bar X_{W,t},t)\ \subset \P^2_\C\times\C, 
\!\!\!\!
\end{equation}
we obtain a proper map, i.e. a ``compactification'' of \eqref{eq:5}:
\begin{equation}
\label{eq:7}
\bar f_W\ :\ \bar{\mathbf{X}}_W\ \longrightarrow\ \C.
\end{equation}
However, the spaces $\bar X_{W,t}$ and $\bar{\mathbf{X}}_W$ are not manifolds with
 boundary (note that their ``boundaries'' $\P^1_\C$ and $\P^1_\C\times
 \C$, respectively, have the same dimension as the ``interior''
 $X_{W,t}$ and $\mathbf{X}_W$). 

By a lack of tools to handle such objects at
 present, we shall not use this compactification  in the present
 paper. 
Nevertheless, in the following Theorem 3, we  show that $f_W$ induces a locally topologically trivial fibration over $\C\setminus\{0,1\}$.
The proof is an elementary handwork, however it is not standard due
to the transcendental nature of  $f_W$ mentioned. Therefore,
we write the proof down to the earth fasion.
\end{remark}

\begin{theorem} For each type $W\in\{\AA,\DD\}$, we have the followings.

\smallskip
{\bf 1.}  The function  $f_{W}$ has only two critical values 0 and 1. That is,
the set of critical points $C_{W}$ of $f_W$ is contained in two fibers $X_{W,0}$ and $X_{W,1}$.

{\bf 2.} {\bf i)} The critical set $C_W$  lies in the real form $\R^2$ of $\mathbf{X}_W$. 

{\bf ii)}  The Hessian form of $f_W|_{\R^2}$ at a critical point is non-degenerate. More precisely, the Hessian form is indefinite at a point in $C_{W,0}\!:=\!C_{W}\!\cap\! X_{W,0}$ and is negative definite at a point in $C_{W,1}\!:=\!C_{W}\!\cap\! X_{W,1}$. 

\smallskip
{\bf iii)} We have the natural bijections:
\begin{eqnarray}
\label{eq:8}
\qquad \qquad \qquad C_{W,0} &\! \simeq\! & \{\text{nodes of type \ } P\}\quad
(\text{identity map}),\\
\label{eq:9}
C_{W,1} &\! \simeq\! & \{\text{bcc's of  type \ } W\}\quad (c \mapsto  B_c :=
\substack{\text{\normalsize\!\!\! the bcc \quad }\\ \text{\normalsize containing } }\ c)
\ \ 
\end{eqnarray}

\smallskip
{\bf 3.} The restriction of the map $f_W$ to the smooth fibers:
\begin{equation}
\label{eq:10}
f_{W}|_{\mathbf{X}_W\setminus(X_{W,0}\cup X_{W,1})}: \mathbf{X}_W\setminus(X_{W,0}\cup X_{W,1})\rightarrow \C\setminus\{0,1\}
\end{equation}
is a topologically locally trivial fibration.
\end{theorem}

\begin{proof} {\bf 1.} We proceed direct calculations separately for each type.

\smallskip\noindent
{\bf $\AA$:} 
The defining equations for $C_\AA$ are $\partial_xf_\AA\!\!=\!cs\!=\!0, \partial_yf_\AA\!\!=\!-2y\!=\!0$. Hence, $C_\AA\!=\!\{(x,0) \mid s(x)\!=\!0 \text{ or } c(x)\!=\!0\}$, where we have 
\vspace{-0.2cm}
\[
f_\AA(x,0) =
\begin{cases}
0 \quad \text{ if } s(x)=0,\\
1\ \quad \text {if } c(x)=0.
\end{cases}
\]
{\bf $\DD$:} 
The defining equations for $C_\DD$ are $\partial_xf_\DD\!\!=\!cs-y^2\!=\!0, \partial_yf_\DD\!\!=\!-\!2xy\!=\!0$. Hence, $C_\DD\!\!=\!\{(0,\pm1)\}\cup\{ (x,0)\mid s(x)\!=\!0 \text{ or }c(x)\!=\!0\}$, where we have 
\vspace{-0.2cm}
\[f_\DD(0,\pm1)\!=\!0 \quad \text{and} \quad 
f_\DD(x,0) =
\begin{cases}
0 \quad \text{ if } s(x)=0,\\
1\ \quad \text {if } c(x)=0.
\end{cases}
\]
\smallskip\noindent
{\bf 2.} i)  Due to the descriptions of $C_W$ in {\bf 1.}, we have only to show that the zero loci of  $s(x)\!=\!0$ and $c(x)\!=\!0$ are real numbers. This follows from the fact that the infinite product expressions \eqref{eq:3} and \eqref{eq:4} are absolutely convergent and the zero loci of $s(x)\!=\!0$ and $c(x)\!=\!0$ are given by the union of zero locus of factors of the expressions, respectively.

\smallskip\noindent
ii)   Let us calculate the Hessian at a critical point. 

The statement for the two critical points $(0,\pm1)$ on $X_{\DD,0}$ can be verified directly.  The other critical points are  on the $x$-axis, i.e.\ one always has $y=0$. Since  $\partial_x\partial_y f_W\mid_{y=0}=0$ for each type $W\in\{\AA.\DD\}$, the Hessian is a diagonal matrix of the form 
\[
[\partial_x(c(x)s(x)), -2 ]_{diag} \quad \text{ for type }  P=\AA ,
\]
\[
[\partial_x(c(x)s(x)), -2x]_{diag} \quad \text{ for type }  P=\DD,
\]
where the second diagonal component is always negative. We calculate the sign of the first diagonal component by 

\centerline{$\partial_x(c(x)s(x))\mid_{c=0}=-\frac{1}{2}s^2=-\frac{1}{2x}<0$ and $\partial_x(c(x)s(x))\mid_{s=0}=\frac{1}{2x}>0$, 
}
\noindent
implying the statement {\bf ii)}. 

\smallskip\noindent
iii)  Combining the explicit descriptions of the set $C_{W,0},\ C_{W,1}$
 in Proof of {\bf 1.}\ with Observations 2.\ and 3.\ in  \S2.2, the
 correspondences are defined and are injective (see Figure 1 and
 2.). So, we need only to show their surjectivity. But, this is
 again trivial since i) any node of a curve is a critical point of the defining equation of the curve, where Hessian is indefinite, and  ii) inside of any bounded connected component of a complement of a real curve in $\R^2$, there exists at least a point where $f_W$ takes local maximum, then the Hessian at the point should be negative definite since we saw in {\bf 2. ii)} that it is already non-degenerate.

\smallskip\noindent
{\bf 3.} 
Let us show that the fibration \eqref{eq:10} is locally topologically
 trivial. Since our map is neither proper nor extendable to a suitably
 stratified proper map (recall Remark 1.3.), we cannot use standard technique such as Thom-Ehreshmann theorems. Instead, we use an elementary fact that $X_{W,t}$ is a ramified covering space: namely,
in view of the equations \eqref{eq:8} and \eqref{eq:9}, the projection map $(x,y)\in\C^2\mapsto x\in \C$ to the $x$-plane induces a proper and ramified double covering maps  $\pi_{W,t}$: 
\begin{equation}
\label{eq:11}
X_{\AA,t}\!\to \C  \ (t\!\in\!\C)  \text{\ \ and\ \ }   X_{\DD,t}\!\to \C\setminus\!\{0\} \ (t\!\in\! \C\setminus\!\{0\}),
\end{equation}
(for $X_{\DD,0}$, see \footnote{
Since the fiber $X_{\DD,0}$ contains an irreducible component $L\!:=\!\{x\!=\!0\}$, the map on $X_{\DD,0}$ is not a covering, but its restriction to $X_{\DD,0}\setminus L$ is a covering.
}).
Let us denote by $\C_W$ the base space of this covering, i.e.\ $\C_W:=\C$ if $W=\AA$ and $:=\C\setminus\{0\}$ if $W=\DD$. 
  In view of the defining equation of $X_{W,t}$, the covering is ramifying at $X_{W,t}\cap\{y\!=\!0\}$, i.e.\ at solutions $x\in \C_w$ of the equation 
\begin{equation}
\label{eq:12}
xs^2(x)-t\ =\ 0,
\end{equation}
which, apparently, has infinitely many solutions, depending on $t\in \C$. 

\medskip
We, now, state an elementary but a crucial fact on the function
 $xs^2$.

\medskip
\noindent
{\bf Fact.} {\it The correspondence $\pi:\C_W\to \C, \  x\mapsto t\!:=\!xs^2(x)\!=\!\mathrm{sin}^2(\sqrt{x})$ is ramifying exactly and only at the inverse images of the points $0$ and $1$, and induces a (topological) covering map over $\C\setminus\{0,1\}$.}

\medskip
\noindent
{\it Proof} of  {\bf Fact.} 
The critical points of the map $t\!=\!xs^2(x)$ are given by the equation
 $s(x)c(x)=0$, and are exactly the points where $t\!=\!0$ or $1$)
 (recall Proof of {\bf 1.}). Thus, the restricted map
 $\pi':=\pi|_{\pi^{-1}(\C\setminus\{0,1\})}$ over $\C\setminus\{0,1\}$
 is a locally homeomorphism. To see that $\pi'$ is a covering (i.e.\ a
 proper map  on each component of an inverse image of a simply connected
 open subset of $\C\setminus\{0,1\}$), we need to show that the inverse map of $xs^2(x)\!=\!t$ as a
 multivalued function in $t$ is analytically continuable everywhere on
 the set $\C\setminus\{0,1\}$.  Since the equation is equivalent to
 $\sqrt{x}\!=\!\pm\sin^{-1}(\sqrt{t})$, this fact follows from the fact
 that the multivalued function $\sin^{-1}(u)$ has singular points (i.e.\
 points where
 the function cannot be analytically continued) only at $u\!=\!\pm1$,  easily seen from the integral expression 
$\sin^{-1}(u)\!=\! \int^u_0\frac{du}{\sqrt{1-u^2}}$. 
$\Box$

\smallskip
Owing to {\bf Fact}, 
 we find  a disc neighbourhood $\U$ for any $t_0\!\in\!
 \C\setminus\!\{0,1\}$ so that $\pi^{-1}(\U)$ decomposes into components
 homeomorphic to $U$. For each $x_i\!\in\! \pi^{-1}(t_0)$ ($i\in I$
 index set), let $s_i(t)$ be the function on $t\!\in\! \U$, defining a section of $\pi$ such that $s_i(t_0)\!=\!x_i$ (actually, $s_i(t)\!=\!\big(\sqrt{x_i}\!+\!\int_{\sqrt{t_0}}^{\sqrt{t}}\frac{du}{\sqrt{1-u^2}}\big)^2$ for choices of $\sqrt{t_0}$ and $\sqrt{x_i}$ such that $\sqrt{t_0}\!=\!\sin{(\!\sqrt{x_i})}$ and path of integral in the connected component of $\pm\sqrt{\U}$ containing $\sqrt{t_0}$).

We can find a differentiable map $\varphi:\U\!\times\! \C_W\!\to\! \C_W$ such that i) $\varphi(t_0,x)\!=\!x$, ii) for each $t\!\in\! U$, the $\varphi_t:=\varphi(t,\cdot)$ is a diffeomorphism of $\C_W$, and iii) for each $i\in I$, $\varphi(t,s_i(t))$ is constant (equal to $s_i(t_0)\!=\!x_i$). 
The diffeomorphism $\varphi_t$ can be uniquely lifted to a diffeomorphism $\hat\varphi_t:X_{W,t}\simeq X_{W,t_0}$ of the double covers such that $\varphi_t\circ \pi_{W,t}=\pi_{W,t_0}\circ \hat\varphi_t$. The $\hat\varphi_t$ gives the local trivialization of \eqref{eq:10}. 
\ \ $\Box$

This completes a proof of Theorem 1., 2. and 3.
\end{proof}

\section{Vanishing cycles  of types $\AA$ and $\DD$}

We show that the middle homology group of a generic fiber of the map
\eqref{eq:5} has basis consisting of vanishing cycles. The intersection
form among them forms the {\it principal quiver}\footnote
{ We mean by a {\it quiver} an oriented graph. It is called {\it principal}, if the set of vertices's has a bipartite decomposition $\Gamma_0\sqcup \Gamma_1$ such that the head (resp. tail) of any edge belongs to  $\Gamma_0$ (resp. $\Gamma_1$) (e.g.\ Figure 3 and 4).  See [Sa2,3].}
of type $\AA$ or $\DD$. 

\subsection{Middle homology groups }

In the present paragraph, we describe the middle homology group of the general fibers
of \eqref{eq:10} in terms of vanishing cycles of the function
$f_W$ of type {\small $W\in\{\AA,\DD\}$}.

\smallskip 
{\bf Vanishing cycles:} For a critical point $c\in
C_W=C_{W,0}\sqcup C_{W,1}$, we define an oriented 1-cycle
$\gamma_{W,c}$  in $X_{W,t}$ for  $t\in (0,1)$ as follows.

Due to Theorem 2, we can choose holomorphic local coordinates $(u,v)$ in
a neighborhood $\U$ of $c$ in $\mathbf{X}_W$ such that i) $u$ and $v$
are real valued on $\U_\R:=\U\cap\R^2$, ii)
$\frac{\partial(u,v)}{\partial(x,y)}|_{\U_R}\!>\!0$ and iii)
$f_W|_{\U}\!=\!u^2-v^2$ if $c\in C_{W,0}$ and
$f_W|_{\U}\!=\!1\!-\!u^2-v^2$ if $c\in C_{W,1}$.  Then, define cycles:
\begin{equation}
\label{eq:13}
 \gamma_{W,c}:=\!\begin{cases}  (\sqrt{t}\cos(\theta),
  \sqrt{-1}\sqrt{t}\sin(\theta))\  (0\le\theta\le2\pi), \text{ if }
  c\!\in\! C_{W,0}\\ 
(\sqrt{1\!-\!t}\cos(\theta),
 \sqrt{1\!-\!t}\sin(\theta))\  (0\!\le\!\theta\!\le\!2\pi),  \text{ if }
  c\!\in\! C_{W,1}. 
\end{cases}\!\!\!\!\!\!\!\!\!\!\!\!\!\!\!
\end{equation}

\noindent
{\bf Fact.}\ {\it The oriented cycle $\gamma_{W,c}$ in the surface $X_{W,t}$ is, up to free homotopy, unique and independent of a choice of coordinates $(u,v)$. }

\smallskip
\noindent
{\bf Definition.}\ We shall denote the homology class in
$\mathrm{H_1}(X_{W,t},\Z)$ of the cycle $\gamma_{W,c}$ by the same 
$\gamma_{W,c}$, and call it the {\it vanishing cycle} of the function
$f_W$ at the critical point $c\in C_W$ {\small (vanishing along the path $t\!\downarrow\! 0$ or $t\!\uparrow\! 1$)}.

\medskip
\noindent
{\bf Sign convention} of intersection numbers 
of 1-cycles on\!  $X_{W,t}$.

\smallskip
\noindent
{\bf i)} Let $I$ be the skew symmetric intersection form between two 
oriented 1-cycles on a oriented surface. Then we define the convention of the sign of intersection number locally as follows: 
\[\begin{array}{lllllllll}
  & \searrow \   \!\nearrow^{\ \gamma_2}&\quad& & \searrow \ \nearrow^{\ \gamma_1}&
 \vspace{-0.2cm}
\\
 \noindent
{\bf Fig.3}\qquad I(\gamma_1,\gamma_2)=1\ \text{if}& \ \ \times &\ , \ &I(\gamma_1,\gamma_2)=-1\ \text{if}& \ \ \times\!  &
\vspace{-0.2cm}
\\
  &\! \nearrow \   \searrow_{\ \gamma_1}&\quad && \nearrow \   \searrow_{\ \gamma_2}&\\
 \end{array}
\] 
{\bf ii)} The orientation of the surface $X_{W,t}$ is {\footnotesize $\sqrt{-1}$}$dz\!\wedge\! d\bar{z}\!=\!2dx\!\wedge\! dy$ for a local holomorphic coordinate $z\!=\!x\!+\!iy$ on $X_{W,t}$.\
{\it Eg.}\ Cycles $\gamma_x$ and $\gamma_y$ locally homotopic to $x$-axis and $y$-axis intersects with $I(\gamma_x,\gamma_y)\!=\!1$ at $z\!\!=\!\!0$.

\smallskip

\begin{theorem}\!
{\bf 4.}  The middle  homology group of $X_{W,t}$,\ $t\!\in\! (0,\!1)$ is given by
\begin{equation}
\label{eq:14}
\mathrm{H}_1(X_{W,t},\Z) \ \simeq\ \mathrm{H}_W : = \ \mathrm{H}_{W,0}\ \oplus \ \mathrm{H}_{W,1},
\vspace{-0.1cm}
\end{equation}
where
\vspace{-0.2cm}
\begin{eqnarray}
\label{eq:15}
 \mathrm{H}_{W,0}& := & \oplus_{c\in C_{W,0}}\Z \gamma^{}_{W,c} \\
\label{eq:16}
 \mathrm{H}_{W,1}& := & \oplus_{c\in C_{W,1}}\Z \gamma^{}_{W,c} 
\vspace{-0.2cm}
\end{eqnarray}
\vspace{-0.2cm}
are formally defined free abelian group spanned by vanishing cycles.  

\smallskip
\smallskip
{\bf 5.}
Let $I_W: \mathrm{H}_1(X_{W,t},\Z) \times \mathrm{H}_1(X_{W,t},\Z) \to \Z$ 
be the intersection form on the middle homology group. Then we have 
\begin{equation}
\label{eq:17}
I_W\ = \ J_W -\  ^t\!J_W
\end{equation}
where $J_W$ and  $^t\!J_W$ are integral bilinear forms on $\mathrm{H}_{W}$ given by
\vspace{-0.3cm}
\begin{equation}
\label{eq:18} 
J_W(\gamma^{}_{W,c}, \gamma^{}_{W,c'}):=
\begin{cases}
1  & \text{ if } c=c',\\
-1  & \text{ if } c\in C_{W,0},\ c'\in C_{W,1} \text{ and } c\in \overline{B}_{c'},\\
0 & else,
\end{cases}
\vspace{-0.3cm}
\end{equation}
and
\vspace{-0.3cm}
\begin{equation}
\label{eq:19} 
{}^t\!J_W(\gamma^{}_{W,c}, \gamma^{}_{W,c'}):=
\begin{cases}
1  & \text{ if } c=c',\\
-1  & \text{ if } c\in C_{W,1},\ c'\in C_{W,0} \text{ and } c'\in \overline{B}_{c},\\
0 & else.
\end{cases}
\end{equation}
\end{theorem}

\medskip
\noindent
{\bf Remark.} The meaning to use the form $J_W$ shall be clarified in \S3.3.
\begin{proof}
We first calculate intersection numbers between vanishing cycles
 $\gamma_{W,c}$ and $\gamma_{W,c'}$ as  given in  {\bf 5.}

Suppose both critical points $c,c'$ belong to $C_{W,0}$ (resp.\ $C_{W,1}$). If $c\!\not=\!c'$ then we, for $t$ close enough to $0$ (resp.\ $1$),  the supports of the vanishing cycles are close to $c$ and $c'$ so that they are disjoint, i.e.\ $\gamma_{W,c}\!\cap\! \gamma_{W.c'}\!=\!\emptyset$ and we get $I_W(\gamma_{W,c},\gamma_{W,c'})\!=\!0$. Then, this equality holds for any $t\!\in\!(0,1)$. 
If  $c\!=\!c'$, then $I_W(\gamma_{W,c},\gamma_{W,c})=0$ due to skew-symmetry of $I_W$.

Next, we consider a cycle $\gamma_{W,c}$ for $c\!\in\! C_{W,0}$ and a cycle $\gamma_{W,c'}$ for $c'\!\in\! C_{w,1}$.
From their expressions in \eqref{eq:13}, we observe the following two facts:

\smallskip
i) The cycle $\gamma_{w,c}$ intersects only with each of connected component of $\R^2\!\setminus\! X_{W,0,\R}$ adjacent to $c$ at one point  $(u,v)=(\varepsilon\sqrt{t},0)$ for $\varepsilon\in\{\pm1\}$.

ii) The underlying set $|\gamma_{W,c'}|$ is presented by a circle of radius $1\!-\!t$ in the bcc $B_{c'}$ containing $c'$, i.e.\ it is equal to $\{(u',v')\!\in\! B_{c'}\mid f_W(u',v')\!=\!t\}$. 

\smallskip
These means that  cycles $\gamma_{W,c}$ and $\gamma_{W,c'}$ for the same $t\in(0,1)$ intersect  if and only if the critical point $c$ is adjacent to the bounded component $B_{c'}$, and, then, they intersect transversely at one point, say $p$. 
Let  $(u',v')$ be the coordinates for the cycle $\gamma^{}_{W,c'}$ in \eqref{eq:13}. Then, by an orientation preserving orthogonal linear transformation of the coordinates, the intersection point $p$ may be given by $(u',v')=(\sqrt{1\!-\!t},0)$

We determine the sign of the intersection as follows: 
in a neighbourhood of $p$, we have an equality $f_W\!=\!u^2\!-\!v^2\!=\!1\!-\!u'^2\!-\!v'^2$. Then the differentiation at $p$ of the equation gives $df|_p\!=\!\varepsilon\sqrt{t}du|_p\!=\!-\!\sqrt{1\!-\!t}du'|_p$. Since $du\wedge dv|_p=c du' \wedge dv'|_p$ for some positive $c\in\R_{>0}$, we get 

\smallskip
a)
\centerline{
$\frac{\partial v}{\partial v'}|_p\ =\ \varepsilon c\frac{\sqrt{t}}{\sqrt{1-t}}$. }

\smallskip
On the other hand, since $du$ and $d u'$ are co-normal vectors to $X_{W,t}$ at $p$ (i.e.\ $df|_p\ /\!/\ du|_p\ /\!/\ du'|_p$), we  use $dv$ and $d v'$ as for complex coordinates of the 1-dimensional complex tangent space $T(X_{W,t})_p$ at $p$, which are compatible with the sign convention ii) of the surface $X_{W,t}$. 

Using these coordinates, the infinitesimal direction $\frac{\partial}{\partial \theta}|_p$ of $\gamma^{}_{W,c}$ at $p$ is evaluated by 

b)
\centerline{
$\frac{\partial v}{\partial \theta}|_p\ =\ \varepsilon\sqrt{-1}\sqrt{t}$ 
}

\smallskip
\noindent
and the infinitesimal direction $\frac{\partial}{\partial \theta'}|_p$ of $\gamma^{}_{c',1}$ at $p$ is evaluate by

\smallskip
c) 
\centerline{
$\frac{\partial v'}{\partial \theta'}|_p\ =\ \sqrt{1-t}$. 
}

\smallskip
\noindent
Combining a), b) and c), we obtain that the angle from the cycle $\gamma^{}_{W,c'}$ to the cycle $\gamma^{}_{W,c}$ at their intersection point $p$ is given by the angle of the complex number 

\smallskip
d)
\centerline{
$\big(\frac{\partial v}{\partial \theta}|_p / \frac{\partial v'}{\partial \theta'}|_p\big)\ /\ \frac{\partial v}{\partial v'}|_p\ =\ \frac{\sqrt{-1}}{c}$, 
}

\smallskip
\noindent
i.e.\ the angle is $\frac{\pi}{2}$. Then due to our sign convention, we obtain
\[
I_W(\gamma^{}_{W,c},\gamma^{}_{W,c'})= -1 \quad \text{and} \quad
I_W(\gamma^{}_{W,c'},\gamma^{}_{W,c})= 1,  
\]
which is independent of the sign $\varepsilon\in\{\pm1\}$. Thus, \eqref{eq:17} is shown.

\medskip
Finally in the following i)-v), we prove {\bf 4.}  

We formally put \eqref{eq:15} and \eqref{eq:16}.

\noindent
i) Let us first show a natural isomorphism.
\begin{equation}
\label{eq:20}
\mathrm{H}_1(X_{W,0},\Z)\  \simeq\ \mathrm{H}_{W,1}. 
\end{equation}
{\it Proof} of \eqref{eq:20}.  We first show that $X_{W,0,\R}$ is a deformation retract of $X_{W,0}$. 
For the proof of it, recall the double cover expression of $X_{W,0}$ over $\C_w$, used in the proof of {\bf Theorem 3}. In case of type $W=\AA$, the deformation retract of the plane $\C_W$ to the half real axis $\R_{\ge0}$ induces the retract of the covering space $X_{W,0}$ to its real form $X_{w,0,\R}$. In case of type $W=\DD$, we do the retraction irreducible-componentwisely to the real axis $\R$ (details are left to the reader). Thus, in view of Figure 1 and 2, we have a natural isomorphism:
\[
\qquad\qquad  \qquad\qquad \mathrm{H}_1(X_{W,0},\Z)\ \simeq\ \mathrm{H}_1(X_{W,0,\R},\Z)\ \simeq\ \mathrm{H}_{W,1}. \qquad\qquad  \Box \ )
\]

\noindent
 ii) Using the double cover expressions of fibers $X_{W,t}$ in the proof of {\bf Theorem 3.}, we can show that $f_W^{-1}([0,t])$ ($t\in (0,1)$) retracts to its subset $X_{W,0}$. Then composing with the inclusion map $X_{W,t}\subset f_W^{-1}([0,t])$, we get an exact sequence
\[
\mathrm{H}_{W,0}\ \rightarrow\ \mathrm{H}_1(X_{w,t},\Z)\ \overset{r}{\rightarrow}\ \mathrm{H}_1(X_{W,0},\Z)\ \rightarrow\ 0,
\]
where the restriction of $r$ to the submodule $H_{W,1}$ composed  with the isomorphism \eqref{eq:20} induces the identity on $H_{W,1}$.
This implies that $\mathrm{H}_{W,1}$ is a factor of $\mathrm{H}_1(X_{W,t},\Z) $. 

\smallskip
\noindent
iii)
What remains to show is that $\mathrm{H}_{W,0}$ is injectively embedded in $\mathrm{H}_1(X_{W,t},\Z) $. This can be partially shown by using the non-degeneracy of the intersection relations \eqref{eq:18} as follows.

Let $\gamma\!\in\! \mathrm{H}_{W,0}$ be a non-zero element, whose image in  $\mathrm{H}_1(X_{W,t},\Z) $ is zero. Then solving the relation $I_W(\gamma,\gamma_{W,c})\!=\!0$ for $c\!=\!c_{W,1}^{(n)}\!\in\! C_{W,1}$ (see Notation in \S3.2) from large enough $n\!\in\!\Z_{>0}$ back wards to 1, we see successive vanishings of the coefficients of $\gamma$, and finally see that $\gamma$, up to a constant factor, is equal to $\gamma_{D,0}^+\!-\! \gamma_{D,0}^-$ (see \S3.2 for Notation $\gamma_{D,0}^+$ and $\gamma_{D,0}^-$).
 In order to show that this is not possible, we prepare a fact.

\medskip
\noindent
iv) {\bf Fact.}\ {\it The function $f_W$ of type $W$ is invariant by the involution $\sigma\!:\!\mathbf{X}_W\!\to\! \mathbf{X}_W,\ (x,y)\!\mapsto\! (x,-\!y)$ on its domain, i.e.\ $f_W\circ\sigma\!=\!f_W$. The induced involution on the surface $X_{W,t}$, denoted again by $\sigma$, is equivariant with the covering map $\pi_{W,t}$ \eqref{eq:11}, i.e.\ $\pi_{W,t}\circ\sigma=\pi_{W,t}$. 
Then, one has
} 
$\sigma_*(\gamma_{W,c}) =-\gamma_{W,c} $
{\it for all $c\in C_W$, except for the following two cases }
\[
\sigma_*(\gamma_{D,0}^+)=-\gamma_{D,0}^-  \ \text{
 and }\ 
\sigma_*(\gamma_{D,0}^-)=-\gamma_{D,0}^+ .
\]
{\it Proof} of {\bf Fact.} Except for the  cases $\gamma_{D,0}^+$ and $\gamma_{D,0}^-$, we can choose the
 coordinate in \eqref{eq:13} in such manner that $\sigma(u,v)=(u,-v)$. \qquad  $\Box$

\medskip 
\noindent
v) Assuming $\gamma_{D,0}^+\!=\! \gamma_{D,0}^-$, let us show a contradiction. Consider the homomorphism $(\pi_D)_*: \mathrm{H_1}(X_{D,t},\Z)\to \mathrm{H_1}(\C_{D},\Z)\simeq\Z$.
Above {\bf Fact.}\ implies $(\pi_{D})_*(\gamma_{D,0}^+)\!=\! (\pi_{D}\circ\sigma)_*(\gamma_{D,0}^+)\!=\! (\pi_{D})_*\circ\sigma_*(\gamma_{D,0}^+)\!=\!- (\pi_{D})_*(\gamma_{D,0}^-)$ which, by the assumption, is equal to  $-(\pi_{D})_*(\gamma_{D,0}^+)$. Thus, we get $(\pi_{D})_*(\gamma_{D,0}^+)\!=\!0$. This contradicts to the fact that $(\pi_{D})_*(\gamma_{D,0}^+)$ generates $\mathrm{H_1}(\C_{D},\Z)\!\simeq\!\Z$ (observed easily from the fact that the equation $x\!=\!0$ defines i) a branch of $X_{D,0,\R}$ at the nodal point $c_{D,0}^+$ and also ii) the puncture in $\C_D$, and from the description of $\gamma_{D,0}^+$ in \eqref{eq:13}).

\smallskip
This completes a proof of Theorem 4. and 5.
\end{proof} 
\begin{remark}
In the step v) in above proof, we may use 
a $\sigma$-invariant form $\omega\!:=\!\mathrm{Res}\big[\frac{ydxdy}{f_D-t}\big]$. Since $\int_{\gamma_{D,0}^+}\!\omega\!=\!\int_{\gamma_{D,0}^+}\!\sigma^*(\omega)\!=\!\int_{\sigma_*(\gamma_{D,0}^+)}\!\omega\!=\!-\!\int_{\gamma_{D,0}^-}\!\omega$, the assumption $\gamma_{D,0}^+\!=\! \gamma_{D,0}^-$ implies $\int_{\gamma_{D,0}^+}\omega\!=\!0$. On the other hand, $\omega\!=\!\mathrm{Res}\big[\frac{ydxdy}{f_D-t}\big]\!=\!\frac{dx}{2x}|_{X_{D,t}}$, and hence $\int_{\gamma_{D,0}^+}\omega\!=\!\pm$${\tiny \sqrt{-1}}$$\pi\not=0$. A contradiction!

\end{remark}

\subsection{Quivers of type {\footnotesize $\AA$} and {\footnotesize $\DD$}}.

\noindent
We encode homological data of vanishing cycles of $f_W$ in a quiver $\Gamma_W$. 

\medskip
\noindent
{\bf Definition.} A quiver $\Gamma_W$ of type $W\in\{\AA,\DD\}$ is defined by

i) The set of vertices of $\Gamma_W$ is in bijective with $\{\gamma^{}_{W,c}\mid c\in C_{W,0}\cup C_{W,1}\}$.

ii) We put an oriented edge from $\gamma^{}_{W,c}$ to $\gamma^{}_{W,c'}$ if and only if $c\!\in\! C_{W,0},$ $ c'\!\in\! C_{W,1}$ and $c\in\overline{B}_{c'}$, that is, when $J_W(\gamma^{}_{W,c},\gamma^{}_{W,c'})=-1$.

\medskip
 Let us fix a numbering of elements in $C_{W,0}\cup C_{W,1}$ as follows.
\begin{eqnarray*}
C_{A,0}\!&\! =\!&\! \{ 
c_{A,0}^{(n)}:=
(n^2\pi^2,0)\}_{n\in\Z_{>0}} \\
C_{A,1}\!&\! =\!&\! \{ 
 c_{A,1}^{(n)}:=
((n-\frac{1}{2})^2\pi^2,0)\}_{n\in\Z_{>0}} \\
C_{D,0}\!&\! =\!&\! \{ c_{D,0}^{(n)}\!:=\!(n^2\pi^2,0)\}_{ n\in\Z_{>0}} 
\ \cup \ 
\{
 c_{D,0}^+\!:=\!
(0,1),  
 c_{D,0}^-\!:=\!
(0,\!-\!1)\} 
 \\ 
%
C_{D,1}\!&\! =\!&\! \{ 
 c_{D,1}^{(n)}:=
((n-\frac{1}{2})^2\pi^2,0) \}_{n\in\Z_{>0}}. 
\end{eqnarray*}
\vspace{-0.2cm}

\noindent
 According to them, the vertices of the quiver $\Gamma_W$ are numbered as below.

\[
{\large \Gamma_{\AA}:}\quad \ \gamma_{A,1}^{(1)} \longrightarrow \gamma_{A,0}^{(1)} \longleftarrow 
\gamma_{A,1}^{(2)} \longrightarrow \gamma_{A,0}^{(2)} \longleftarrow \gamma_{A,1}^{(3)} \longrightarrow \gamma_{A,0}^{(3)} \longleftarrow \quad \cdot\ \cdot\ \cdot \qquad
\]
\[
\begin{array}{llll}
&\gamma_{D,0}^{+}&\\
\vspace{-0.2cm}
&&\!\!\nwarrow\\ 
\vspace{-0.2cm}
\!\!\!\!{\large \Gamma_{\DD}:}& &\quad \gamma_{D,1}^{(1)} \longrightarrow \gamma_{D,0}^{(1)} \longleftarrow 
\gamma_{D,1}^{(2)} \longrightarrow \gamma_{D,0}^{(2)} \longleftarrow \gamma_{D,1}^{(3)} \longrightarrow \ \cdot\ \cdot\ \cdot  \\
\vspace{-0.2cm}
&&\!\! \swarrow\\
\vspace{-0.2cm}
&\gamma_{D,0}^{-}
\end{array}
\vspace{0.2cm}
\]
Note that the decomposition of the critical set $C_W$ into  $C_{W,0}\cup C_{W,1}$ gives arise the bi-partite (or principal) decomposition of the quiver $\Gamma_W$.

%

\subsection{Suspensions to higher dimensions}.

In this subsection, we briefly describe the suspensions of the results in previous subsections to higher dimensional cases.

 For a type $W\!\in\!\{\AA,\DD\}$ and $d\!\in\!\Z_{\ge1}$, let us
 introduce the $d$-th {\it suspension} $f_W^{(d)}$ of $f_W$ (where
 $f_W^{(1)}\!=\!f_W$) as the entire functions in $d+1$-variables $x,y$ and $\underline z=(z_2,\cdots,z_d)$ defined by
\begin{equation}
\label{eq:21}
f_W^{(d)}(x,y,\underline{z}):= f_W(x,y) - z_2^2 -\cdots- z_d^2.
\end{equation}
Then, replacing the function $f_W$ by $f_W^{(d)}$ and the domain
$\mathbf{X}_W\!=\!\C^2$ by $\mathbf{X}_W^{(d)}\!=\!\C^2\!\times\! \C^{d\!-\!1}$,
we obtain a holomorphic map \eqref{eq:5}$^{(d)}$ whose {\it fibers, denoted by $X_{W,t}^{(d)}$ ($t\!\in\!\C$), are Stein variety of complex dimension $d$}. 

Replacing, further, the real form $\R^2$ of $\mathbf{X}_W$ by the real form $\R^2\!\times\!\R^{d-1}$ of $\mathbf{X}_W^{(d)}$, {\bf Theorem 1., 2., 3.}\ in \S1.3 hold completely parallely for $f_W^{(d)}$, where the set of critical points of $f_W^{(d)}$ is bijective to that of $f_W$ by the natural embedding $\mathbf{X}_W\subset \mathbf{X}_W^{(d)}$ so that we identify them. Then the signature of Hessians of $f_W^{(d)}$ at points of $C_{W,0}$ is $(1,d)$ and that at points of $C_{W,1}$ is $(0,d\!+\!1)$. The suspended fibration shall be referred by \eqref{eq:10}$^{(d)}$.  The proof 
are reduced to the original case $d\!=\!1$. 

Applying $d\!-\!1$-times suspension $S$ on a homology class $\gamma$ in
$\mathrm{H}_1(X_{W,t},\Z)$, we obtain an element $S^{d\!-\!1}\gamma$ of
the middle homology group $\mathrm{H}_{d}(X_{W,t}^{(d)},\Z)$ of the
$d$-dimensional fiber $X_{W,t}^{(d)}$. In particular, the suspension  $S^{d\!-\!1}\gamma_{W,c}$ of a vanishing cycle $\gamma_{W,c}$ of $f_W$ at a critical point $c\in C_W$ is a vanishing cycle of $f_W^{(d)}$ at the same critical point, which, for simplicity, we shall denote again by $\gamma_{W,c}$. Then replacing $\mathrm{H}_1(X_{W,t},\Z)$ by the middle homology group $\mathrm{H}_{d}(X_{W,t}^{(d)},\Z)$, {\bf Theorem 4.}\ in \S2.1 holds completely parallely, where we keep notations \eqref{eq:14} and \eqref{eq:15}.

The intersection form $I_W^{(d)}$ on the middle $(=\!d)$-dimensional homology group is well known to be symmetric or skew-symmetric according as cycles are even or odd dimensional (i.e. according as $d$ is even or odd). It is also well known that $I_W^{(d)}(\gamma_{W,c},\gamma_{W,c})=(-1)^{\frac{d}{2}}2$ for even $d$-dimensional vanishing cycles. Therefore,  the formula \eqref{eq:17} of the intersection form in {\bf Theorem 5.}\  need to be slightly modified as in the following theorem, where we keep the notation $J_W$ and $^t\!J_W$ together with the formulae \eqref{eq:18} and \eqref{eq:19}. 

\bigskip
\noindent
{\bf Theorem 5$^{(d)}$.} {\it Let $I_W^{(d)}: \mathrm{H}_{d}(X_{W,t}^{(d)},\Z) \times \mathrm{H}_{d}(X_{W,t}^{(d)},\Z) \to \Z$ 
be the intersection form on middle-homology groups of the fibers of the
fibration \eqref{eq:10}$^{(d)}$. Then we have the following 4-periodic expression.
\begin{equation}
\label{eq:22}
I_W^{(d)}= (-1)^{[\frac{d}{2}]}J_W-(-1)^{[\frac{d-1}{2}]}\ ^t\!J_W.
\end{equation}
}
The proof of Theorem is standard, and is omitted. Actually, the form $I_W^{(d)}$
is symmetric for $d$ even and is skew symmetric for $d$ odd.

\smallskip
\noindent
{\bf Remark.} We may regard that the form $J_W$ is an infinite rank analogue of a {\it Seifert matrix} with respect to a ``suitable compactification'' of the three-fold $f_W^{-1}(S^1)$, where $S^1$ is a circle in the base space $\C$ of \eqref{eq:5} which encloses the two points $0$ and $1$. However, we do not pursue any further this analogy (see \S1.3 Remark and the next subsection \S2.4).

\subsection{Monodromy Transformations and  Coxeter elements}.

The fundamental group $\pi_1(\C\setminus\{0,1\}, t_0)$ with $t_0\in (0,1)$ of the base space of the fibration \eqref{eq:10}$^{(d)}$  has two generators $g_0$ and $g_1$ which are presented by circular paths in $\C\!\setminus\!\{0,1\}$ starting at $t_0$ and turning once around the point $0$ and $1$ counterclockwise, respectively. 
Let $\sigma_{W,0}^{(d)}$ (resp.\ $\sigma_{W,1}^{(d)}$) be the monodromy action of $g_0$ (resp.\ $g_1$) on the middle homology group \eqref{eq:14}$^{(d)}$ of the fiber of the family \eqref{eq:10}$^{(d)}$, which preserves the intersection form \eqref{eq:22}. Though the singular fibers $X_{W,0}^{(d)}$ and $X_{W,1}^{(d)}$ have infinitely many critical points, we can apply Picard-Lefschetz formula. That is,  for $u\!\in\! H_W:=H_{W,0}\!\oplus\! H_{W,1}$  
\begin{equation}
\begin{array}{llll}
\label{eq:23}
\sigma_{W,0}^{(d)}(u)\! &\!=\!&\! u +(-1)^{[\frac{d-1}{2}]}\!\!\sum_{c\in C_{W,0} } I_W^{(d)}(u,\gamma_{W,c})\gamma_{W,c}  \\
 &\!=\!&\! u + \sum_{c\in C_{W,0} }((-1)^{d\!-\!1}J_w(u,\gamma_{w,c})-J_W(\gamma_{W,c},u))\gamma_{W,c} \\
 &\!=\!&\!\!\begin{cases}
(-1)^{d\!-\!1} u & \text{if \ } u\!\in\! H_{W,0}\\
u-\sum_{c\in C_{W,0} }J_W(\gamma_{W,c},u)\gamma_{w,c}& \text{if \ } u\!\in\! H_{W,1}
\end{cases}
\end{array}
\end{equation}

\begin{equation}
\begin{array}{llll}
\label{eq:24}
\sigma_{w,1}^{(d)}(u)\!\! &\!=\!&\!\! u +(-1)^{[\frac{d-1}{2}]}\!\sum_{c\in C_{W,1}} I_W^{(d)}(u,\gamma_{W,c})\gamma_{W,c}  \\
 &\!=\!&\! u + \sum_{c\in C_{W,1} }((-1)^{d\!-\!1}J_W(u,\gamma_{W,c})-J_W(\gamma_{W,c},u))\gamma_{W,c} ,  \!\!\!\!\!\\
&\!=\!&\!\!\begin{cases}
u+(-1)^{d\!-\!1}\sum_{c\in C_{W,1} }J_W(u,\gamma_{W,c})\gamma_{W,c}& \text{if \ } u\!\in\! H_{W,0} \!\!\! \\
(-1)^{d\!-\!1} u & \text{if \ } u\!\in\! H_{W,1}.\!\!\!
\end{cases}
\end{array}
\end{equation}
Note that $\sigma_{W,0}^{(d)}=\sigma_{W,0}^{(d\!+\!2)}$ and $\sigma_{W,1}^{(d)}=\sigma_{W,1}^{(d\!+\!2)}$ for $d\in\Z_{\ge1}$. 

\smallskip
\noindent
{\it Note.}
Let $d$ be even. Then the reflections $\sigma_{W,0}^{(d)},\
\sigma_{W,1}^{(d)}$ are involutive:
\begin{equation}
\label{eq:25}
(\sigma_{W,0}^{(d)})^2=(\sigma_{W,1}^{(d)})^2 = \id_{H_W} . 
\end{equation}
In the next section, we shall see that the eigenvalues in a suitable sense of the  product
$\sigma_{W,0}^{(d)}\circ\sigma_{W,1}^{(d)}$ is "dense" in the unit
circle $S^1$ in $\C^\times$, and hence
$\sigma_{W,0}^{(d)}\circ\sigma_{W,1}^{(d)}$ is of infinite order. 
As a consequence, there is no more relations among
$\sigma_{W,0}^{(d)}$ and $\sigma_{W,1}^{(d)}$, and the monodromy group is isomorphic to $\Z/2\Z*\Z/2\Z$.

\medskip
\noindent
{\bf Definition.} 
In analogy with the classical simple singularities, let us call the
product of the two monodromy transformations $\sigma_{W,0}^{(d)}$ and
$\sigma_{W,1}^{(d)}$ a {\it Coxeter element}. Two Coxeter elements
depending on the order of the product are conjugate to each other. We fix one order as follows and call the product the Coxeter element.

\vspace{-0.2cm}
\begin{equation}
\begin{array}{rlll}
\label{eq:26}
Cox_W^{(d)}(u)\!& :=\ \sigma_{W,0}^{(d)}\circ\sigma_{W,1}^{(d)}\ (u) \\
&& \\
=\quad &\!\!\!\!\!\!\!\! \begin{cases}
(-1)^{d\!-\!1}\big(u+\sum_{c\in C_{W,1} }J_W(u,\gamma_{W,c})\gamma_{W,c}\\
-\sum_{c\in C_{W,1} }\sum_{d\in C_{W,0} }J_W(u,\gamma_{W,c})J_W(\gamma_{W,d},\gamma_{W,c})\gamma_{W,d}\big) & \text{if } u\!\in\! H_{W,0}\\
\quad\vspace{-0.4cm}\\
(-1)^{d\!-\!1}\big(u-\sum_{c\in C_{W,0} }J_W(\gamma_{W,c},u)\gamma_{W,c}\big)& \text{if } u\!\in\! H_{W,1} .
\end{cases}
\end{array}
\end{equation}
{\bf Observation.}
{\it The Coxeter element is, up to the sign factor $(-1)^{d\!-\!1}$, independent
of the suspensions for $d\in\Z_{\ge1}$ \eqref{eq:21}}.

\medskip
\noindent
{\bf Remark.}  It is wellknown that a classical Coxeter element for a
root system of finite type $W$ is semisimple of finite order, and
$\frac{1}{2\pi \sqrt{-1}}\log$ of its
eigenvalues, referred as {\it spectra} and given by $\frac{m_i}{h}$
{\small ($i\!=\!1,\!\cdots\!,n$)}, play important role in Lie theory
{\small (\cite{Bo})}. They appear also as exponents of the primitive
forms associated with simple polynomials of type $W$ {\small \cite{Sa1}} and the
fact they lie in the interval $(0,1)$ for the case $d=2$ characterize that
they are primitive forms associated with simple polynomials {\small \cite{Sa4}}. 

The Coxeter
elements of types $\AA$ and $\DD$ are no longer of finite order. However, 
in the next section, we show that they are diagnalizable in suitable
sense and   
the {\it spectra} for them are introduced. Then, the sign factor
$(-1)^{d\!-\!1}$ in \eqref{eq:26} of the Coxeter element $Cox_W^{(d)}$ is lifted to the
shift of the spectra by $\frac{d\!-\!1}{2}$ and
of the spectra so that the spectra of $Cox_W^{(d)}$ is contained in the
interval $(\frac{d}{2}-1,\frac{d}{2})$. The spectra should 
play a key role for primitive forms of type $\AA$
and $\DD$ in a forth coming paper, where the shift of the spectra corresponds to the $\frac{d\!-\!1}{2}$-shift of the primitive forms in the semi-infinite Hodge filtration.

\section{Spectra of Coxeter elements of types $\AA$ and $\DD$}
We study spectra of the Coxeter element $Cox_W^{(d)}$ for {\small $W\in\{\AA,\DD\}$}. For the purpose,
we extend the domain of the Coxeter element to the completion of
$H_{W,C}:=H_W\otimes_\Z\C$ with respect to the $l^2$-norm with  the
ortho-normal basis  $\{\gamma_{W,c}\}_{c\in\C_{W}}$. The Coxeter element
action on this space is diagonalizable (in a suitable sense), and its
``eigenvalues'' take values in the unit circle $S^1\subset \C^\times$. 
We want to determine $\frac{1}{2\pi\sqrt{-1}}\log$ of the ``eigenvalues'', called the {\it spectra} of the Coxeter
element. Actually, it  is calculated by a help of the intersection form
$\dot{I}_W$, since it, as a positive symmetric operator, has only
positive real eigenvalues. It turns out that we get a continuous spectra on the interval $(\frac{d}{2}-1,\frac{d}{2})$.  

\subsection{Hilbert space $\overline H_{W,\C}$}.

Consider $\C$-vector spaces obtained by the complexification of the $\Z$-lattices $H_{W,0},\ H_{W,1}$
and $H_W$ (recall \eqref{eq:14},\eqref{eq:15} and \eqref{eq:16}):
\begin{eqnarray}
\label{eq:27}
\quad H_{W,0,\C}\!:=\!\! H_{W,0}\!\!\otimes_\Z \!\C, H_{W,1,\C}\!:=\!\! H_{W,1}\!\!\otimes_\Z\! \C &\! \text{and}\! & H_{W,\C}\!:=\!\! H_{W}\!\!\otimes_\Z\!\C.\!\!\!\!\!\!\!\!\!\!\!\!\!\!\!\!\!\!\!\!\!
\end{eqnarray}
We equip them with a hermitian inner product $\langle\cdot,\cdot\rangle$ defined by 
\begin{equation}
\label{eq:28}
\langle \sum_{c\in C_W} a_c\gamma_{W,c}, \sum_{c\in C_W} b_c\gamma_{W,c}\rangle :=\sum_{c\in C_W} a_c \bar b_{c},
\end{equation}
where $a_c,b_c$ ($c\!\in\! C_W$) are complex numbers. Then, the
$l^2$-completions  of the spaces with respect  to this inner product
are separable Hilbert spaces, denoted by $\overline H_{W,0,\C}$, $\overline H_{W,1,\C}$ and $\overline H_{W,\C}$, respectively. We have the orthogonal direct sum decomposition: 
\begin{equation}
\label{eq:29}
\overline H_{W,\C}= \overline H_{W,0,\C} \oplus\overline H_{W,1,\C}.
\end{equation}
Let us denote by $\pi_0$ and $\pi_1$ the orthogonal projections of the space $\overline H_{W,\C}$ to the subspaces $\overline H_{W,0\C}$ and $\overline H_{W,1,\C}$, respectively, so that the sum 

\centerline{
$id_{\overline H_{W,\C}}\ = \ \pi_0\ +\ \pi_1$ 
}
\noindent
is the identity map on  $\overline H_{w,\C}$.

Remark that the lattice $H_W$\! is self-dual:\! $\Hom_\Z(H_W,\Z)\!\cap\! \overline H_{W,\C}\!=\! H_W$.

\medskip
\noindent
{\bf Convention.} In the sequel of the present paper, we freely
identify a continuous bilinear form $A$ on $\overline H_{W,\C}$ (resp.\
$H_{W,\C}$) and a
continuous endomorphism $\dot{A}$ on $\overline H_{W,\C}$ (resp.\
$H_{W,\C}$) by the
following relations: 
\begin{equation*}
A(\xi,\eta)=\langle \dot{A}(\xi),\eta\rangle  \text{\quad and\quad }
\sum_{c\in C_W}A(u,\gamma_{W,c})\gamma_{W,c}= \dot{A}(u).
\vspace{-0.2cm}
\end{equation*}
Transposes $^t\!A$ of $A$ and $^t(\dot{A})$ of $\dot{A}$ are defined by
the relations $^t\!A(\xi,\eta)\!=\!A(\eta,\xi)$ and $\langle
\dot{A}(u),v\rangle\!=\!\langle u,^t\!(\dot{A})(v)\rangle$, respectively. Then, $^t(\dot{A})=\dot{(^t\!A)}$ .

\subsection{Extendability of  $I_W^{(d)}$ and  $Cox_{W}^{(d)}$ on 
$\overline{H}_W$}.

In order to calculate the eigenvalues  of  the intersection forms $I_W^{(d)}$ and the Coxeter elements $Cox_W^{(d)}$, we use the identification mentioned at the end of \S3.1.
Before we do this, we need to check that they are continuously
extendable to the completion $\overline H_{W,\C}$. This is achieved by
using the extendabilities of the endomorphisms $\dot{J}_W,\
^t\!\dot{J}_W$ associated with the bilinear forms \eqref{eq:18} and
\eqref{eq:19}. Put

\begin{equation}
\begin{array}{llll}
\label{eq:30}
\dot{J}_W(u)& := &\sum_{c\in C_{W}}J(u,\gamma_{W,c})\gamma_{W,c} \\
              &\ =& \begin{cases}
              u+ \sum_{c\in C_{W,1} }J_W(u,\gamma_{W,c})\gamma_{W,c} & \text{ if  } u\in H_{W,0}\\
              u & \text{ if  } u\in H_{W,1}\\
              \end{cases}
              \end{array}
               \end{equation}
\begin{equation}
\begin{array}{llll}
\label{eq:31}
^t\!\dot{J}_W(u)& := &\sum_{c\in C_{W}} {}^t\!J(u,\gamma_{W,c})\gamma_{W,c} \\
              &\ =& \begin{cases}
              u& \text{ if  } u\in H_{W,0}\\
              u + \sum_{c\in C_{W,0} }J_W(\gamma_{W,c},u)\gamma_{W,c} & \text{ if  } u\in H_{W,1}\\
              \end{cases}
              \end{array}
              \end{equation}
which are endomorphisms on $H_{W,\C}$, since the quiver $\Gamma_W$ in \S2.2 is locally finite,
i.e.\ any vertex is connected with only finite number of other
vertexes. 
The inverse action of $\dot{J}_W$ (resp.\ $^t\!\dot{J}_W$) on $H_{W,\C}$
 can be obtained by just replacing ``+'' by ``$-$'' in RHS of
 \eqref{eq:30} (resp.\ \eqref{eq:31}). 
 
\medskip
\noindent
{\bf Assertion 1.} {\it The endomorphisms $\dot{J}_W$, $^t\!\dot{J}_W$
and their inverses $\dot{J}_W^{-1}$, $^t\!\dot{J}_W^{-1}$ acting on $H_{W,\C}$ are extendable to bounded endomorphisms on $\overline H_{W,\C}$. The extensions are transpose to each other.}
\begin{proof}  We show only the extendability of the domain of endomorphisms
 $\dot{J}_W$, $^t\!\dot{J}_W$ and their inverses $\dot{J}_W^{-1}$, $^t\!\dot{J}_W^{-1}$ from $H_{W,\C}$ to
 $\overline H_{W,\C}$, where the extensions are denoted by the
 same notation. Then the
 relations $^t(\dot{J}_W)=^t\!\!\dot{J}_W$,
 $\dot{J}_W\dot{J}_W^{-1}=\id_{H_{W}}$, \ldots, etc. are automatically preserved for
 the extensions. 

The quivers $\Gamma_{\AA}$
and $\Gamma_{\DD}$ show that any critical point $c\in
 C_{W,0}$ is adjacent to at most two bdd components. In view of
 \eqref{eq:30}, this implies the inequality $\parallel
 \!\dot{J}_W(u)\!-\!u\!\parallel\le 2\!\parallel\! u\! \parallel$. Hence
 $\dot{J}_W$ is extendable to a bounded endomorphims on $\overline
 H_{W,\C}$, denoted by the same $\dot{J}_W$. 

 We observe also that,  to any bdd component, at most 3 critical points in $C_{W,0}$ are adjacent (actually, 3 occurs only one bdd component for the critical point $c_{D,1}^{(1)}$ of type $\DD$).  In view of \eqref{eq:31}, we get an inequality $\parallel\ ^t\!\dot{J}_W(u)-u\parallel\le 3\parallel u\parallel$, implying again the extendability of $^t\!\dot{J}_W$ to a bounded endomorphism on $\overline H_{w,\C}$, denoted by the same $^t\!\dot{J}_W$.

 Similar arguments
 shows the extendability of the inverses.
\end{proof}
\noindent
An immediate  consequence of {\bf Assertion 1} is that {\it the endomorphism 
\[
\!\!\!\!{\rm (3.3.{22})}^{\bullet}\qquad\qquad\qquad
\dot{I}_W^{(d)}:= (-1)^{[\frac{d}{2}]}\dot{J}_W-(-1)^{[\frac{d-1}{2}]}\ ^t\!\dot{J}_W\qquad\qquad\qquad
\]
defined on $H_{W,\C}$ is extendable to a bounded endomorphism on 
$\overline{H}_{W,\C}$}. 

Another important consequence of {\bf Assertion 1} is the following.

\begin{corollary}  
 The Coxeter element $Cox_W^{(d)}$ ($d\in \Z_{\ge1}$) defined on $H_{w,\C}$ is
 extendable to an invertible bounded automorphism on $\overline{H}_{W,\C}$.
\end{corollary}
\begin{proof}
 Let us, first, show a formula:
\begin{equation}
\label{eq:32}
Cox_W^{(d)} = (-1)^{d\!-\!1}  \  (^t\!\dot{J}_W)^{-1} \dot{J}_W,
\end{equation}
on $H_W$ by a direct calculation  using formulae \eqref{eq:26},
 \eqref{eq:30} and 
{\small
$$
(^t\!\dot{J}_W)^{-1}(u) = \begin{cases}
              u& \text{ if  } u\in H_{W,0}\\
              u - \sum_{c\in C_{W,0} }J_W(\gamma_{W,c},u)\gamma_{W,c} & \text{ if  } u\in H_{W,1}. \\
              \end{cases}
\leqno (4.2.30)^{-1}   
$$
}
Then, RHS of \eqref{eq:32} is extendable to a bounded operator on
 $\overline{H}_{W,\C}$.

Invertibility of $Cox_W^{(d)}$ follows from that of $\dot{J}_W$ and $^t\!\dot{J}_W$.
\end{proof}

\noindent
{\bf Remark.}  Let $\check{H}_{W,\C}\!:=\!\Hom_\C(H_{W,\C},\C)$ be the (formal) dual vector space
of $H_{W,\C}$. The contragradient actions on $\check{H}_{W,\C}$ of the endomorphisms  
$\dot{J}_W$, $^t\!\dot{J}_{W}$, $\dot{I}_W^{(d)}$,
$^t\!\dot{I}_W^{(d)}$, $Cox_W^{(d)}$ and $^t\!Cox_W^{(d)}$ on $H_{W,\C}$
shall be denoted, as usual, by the super script
``\ $^t$(-)\ '' such that ``\ $^t{}^t$(-)=(-)\ ''.

On the other hand, by regarding $\{\gamma_{W,c}\}_{c\in C_W}$ as the
self-dual basis, $\check{H}_{W,\C}$ is identified with the direct
product $\prod_{c\in C_W}\C \gamma_{W,c}$ so that we have natural inclusions of
$\C$-vector spaces: 
\[
 H_{W,\C}\subset \overline{H}_{W,\C}\subset \check{H}_{W,\C}.
\]
Then it is easy to verify that the extensions of $\dot{J}_W$, $^t\!\dot{J}_{W}$, $\dot{I}_W^{(d)}$,
$^t\!\dot{I}_W^{(d)}$, $Cox_W^{(d)}$ and $^t\!Cox_w^{(d)}$ to the spaces
$\overline{H}_{w,\C}$ and $\check{H}_{W,\C}$ are naturally compatible
with respect to the above inclusions. 
The
relationships between these extensions and the transpositions are given as follows:
\[
{}^t\! \dot{I}_W^{(d)}= (-1)^{d}\dot{I}_W^{(d)} \quad \text{ and }\quad 
(^t\!Cox_W^{(d)})^{-1}= \dot{J}_W\ Cox_W^{(d)}\ \dot{J}_W^{-1} 
. 
\]
However, the bilinear form $I_{W,C}$ itself is no longer extandable to
$\check{H}_{W,\C}$ and the endomorphism $\dot{I}_W$ on
$\check{H}_{W,\C}$ has non-trivial kernel.

\subsection{Spectral decomposition of $I_W^{(d)}$ for even $d$}.

Using the fact \eqref{eq:22}, the bilinear form $I_W^{(d)}$ is symmetric
for even $d$. Let us consider the operator for the cases $d\in\Z_{\ge1}$ with $d\equiv0\bmod 4$,\footnote
{ We choose $d\equiv0\bmod 4$ so that the form $I_W$  is positive
and symmetric, defining
a ``root lattice structure of infinite rank'' on $H_W$ (cf.\ Proof of Assertion 3.).
}
\begin{equation}
\label{eq:33}
\dot{I}_W\ :=\ \dot{I}_W^{(d)}\ =\ \dot{J}_W+\ {}^t\!\dot{J}_W .  
\end{equation}

We, first, determine the point spectrum of the symmetric operator $\dot{I}_W$ on $\overline H_{W,\C}$.
Let us consider following two eigenspaces for $\lambda\in \C$:
\begin{equation}
\label{eq:34}
\check{H}_{W,\lambda}:=
\{\xi\!\in\!\check{H}_{W,\C}\mid
\dot{I}_W(\xi)\!=\!\lambda\xi\} \ \ \text{and}\ \
\overline{H}_{W,\lambda}:=\check{H}_{W,\lambda}\cap \overline{H}_{W,\C}.
\end{equation}

\vspace{0.1cm}
\noindent
{\bf Assertion 2.} 
{\it 
\vspace{-0.1cm}
For each type $W\!\in\!\{\AA,\DD\}$ and all $\lambda\!\in\!\C$, we have 
\begin{equation}
\label{eq:35}
\vspace{-0.1cm}
\dim_\C\check{H}_{W,\lambda}=1 \text{\qquad and\qquad }
\dim_\C\overline{H}_{W,\lambda}=0,
\end{equation} 
\vspace{-0.1cm}
except for the case $W=\DD$ and $\lambda=2$, where we have 
\vspace{-0.1cm}
\begin{equation}
\label{eq:36}
\vspace{-0.2cm}
\dim_\C\check{H}_{\DD,2}=2 \quad\text{ and }\quad
\dim_\C\overline{H}_{\DD,2}=1,
\end{equation}
and $\overline{H}_{\DD,2}$ is spanned by the vector $\gamma_{D,0}^+ -\gamma_{D,0}^-$.
}
\begin{proof} This is shown by solving the equation
 $\dot{I}_W(\xi)\!=\!\lambda\xi$ for the coefficients of
 $\xi\!=\!\sum_{c\in C_W}a_c\gamma_{W,c}\in \check{H}_{W,\C}$ formally
 and inductively according to the following labeling and ordering of coefficients:
\[
\! {\large \Gamma_{\AA}:}\qquad \ a_0 \longrightarrow a_1 \longleftarrow 
a_2\! \longrightarrow a_3\! \longleftarrow a_4\! \longrightarrow a_5 \longleftarrow \quad \cdot\ \cdot\ \cdot \qquad
\vspace{-0.4cm}
\]
\[
\begin{array}{llll}
\vspace{-0.2cm}
&\ \ \ b_{0}^{+}&\\
\vspace{-0.3cm}
&&\!\!\nwarrow\\ 
\vspace{-0.3cm}
\!\!\!\!\!\!\!\!\!\!\!\!\!{\large \Gamma_{\DD}:\ }\ & &\quad b_{1} \longrightarrow b_2 \longleftarrow 
b_3 \longrightarrow b_4 \longleftarrow b_5 \longrightarrow \ \ \ \cdot\
\cdot\ \cdot \! 
\vspace{-0.1cm}\\
\vspace{-0.2cm}
&&\!\! \swarrow\\
\vspace{-0.2cm}
&\ \ \ b_0^{-}
\end{array}
\vspace{0.1cm}
\]
Details of the calculation are omitted.  
Results are summarized as:

\medskip
\noindent
$\AA$: The space $\check{H}_{\AA,\lambda}$ for any $\lambda\!\in\!\C$ is spanned by  
{\small 
\vspace{-0.2cm} 
\[
\check{\xi}_{\AA,\lambda}: \quad a_n=
\frac{\exp((n+1)\sqrt{-1}\pi\theta)-\exp(-(n+1)\sqrt{-1}\pi\theta)}{\exp(\sqrt{-1}\pi\theta)-\exp(-\sqrt{-1}\pi\theta)}
\quad (n\ge0) 
\vspace{-0.1cm} 
\]
}
\noindent
where  $\theta$ is any complex number satisfying
$
 \lambda\! =\! 4 \sin^2(\frac{\pi}{2}\theta).
$
In case $\lambda\!=\!0$ or $4$ (i.e.\ when $\theta\in \Z$), we
 interpret this formula as $a_n=\pm(n+1)$.

\smallskip
\noindent
$\DD$: For all $\lambda\in\C$, let us introduce a vector
{\small
\vspace{-0.05cm} 
\[
\check{\xi}_{\DD,\lambda}: \quad b_0^+=1,\ b_0^-=1,\
 b_n=\exp(n\sqrt{-1}\theta)+\exp(-n\sqrt{-1}\theta) 
      \quad (n\ge1) \!\!\!\!
\vspace{-0.1cm} 
\]
}
where  $\theta$ is any complex number satisfying the equation
$
 \lambda = 4 \sin^2(\frac{\pi}{2}\theta).
$
\vspace{-0.1cm}
Then, the space $\check{H}_{\DD,\lambda}$ for any $\lambda\!\not=\!2$ is spanned 
 by {\small $\check{\xi}_{\DD,\lambda}$}. The space $\check{H}_{\DD,2}$
 is spanned by {\small $\check{\xi}_{\DD,2}$} and 
\[
 \gamma_{D,0}^+-\gamma_{D,0}^-:  \quad b_0^+=1,\ b_0^-=-1,\
 b_n=0  \ \ (n\ge1).
\]

The norm  $\langle\check{\xi}_{W,\lambda},\check{\xi}_{W,\lambda}\rangle$
 \eqref{eq:28} for any $W\!\in\!\{\AA,\DD\}$ and any
 $\lambda\!\in\! \C$ is unbounded, whereas
 $\gamma_{D,0}^+-\gamma_{D,0}^-$ has the bounded  norm=2.
\end{proof}

\noindent
\vspace{-0.15cm} 
{\bf Corollary.}\  {\it The point spectrum of  $\dot{I}_\AA$
on $\overline H_{W,\C}$ is empty, 
and that of $\dot{I}_\DD$ consists of a single eigenvalue
\vspace{-0.15cm}
$\lambda\!=\!2$ with multiplicity $1$. 
In particular, the operator $\dot{I}_W$ for any {\small $W\!\in\!\{\AA,\DD\}$} is 
non-degenerate on $\overline H_{W,\C}$ in the sense that
$\ker(\dot{I}_W)\!\cap\!\overline H_{W,\C}\!=\!\{0\}$.
}

\bigskip
\noindent
{\bf Remark.} By introducing the double cover  of the
$\lambda$-plane by $\mu\!:=\!\exp(\pi\sqrt{-1}\theta)\in
\C\setminus\{0\}$ with the relation 
$2\!-\!\lambda\!=\!\mu\!+\!\mu^{-1}$, the base
$\check{\xi}_{W,\lambda}$ in the proof of {\bf Assertion 2} can be expressed in terms of Laurent polynomials in
$\mu$. The reader may be puzzled by the use of $\theta$ instead of $\mu$
in the above proof.
We used the parameter $\theta$
since it shall parametrize the spectra of Coxeter
elements in the next paragraph. We remark also that
$\lambda\in[0,4]\ \Leftrightarrow \theta\in\R$.

\medskip
For a symmetric operator $\dot{I}_W$  on $\overline H_{W,\C}$, the
{\it greatest lower bound} and the {\it least upper bound}  are defined
as the maximal real number $m$ and the minimal real number $M$
satisfying the following inequalities, respectively (see \cite[\S104]{R-N}).
\begin{equation}
\label{eq:37}
m \langle \xi,\xi \rangle \ \le\ \langle \dot{I}_W(\xi),\xi \rangle=  I_W(\xi,\xi)\ \le \ M\langle \xi,\xi \rangle  \quad \forall \xi\in \overline H_{W,\C}
\end{equation}

\medskip
\noindent
{\bf Assertion 3.} {\it The greatest lower bound $m$ and the least upper
bound $M$ of   $\dot{ I}_W$  for both $W\in\{\AA,\DD\}$ is given 
by $m=0$ and $M=4$.}
\begin{proof} 
For the definition of $m$ and $M$, it is sufficient to run $\xi$ only in
 $H_W$ in the defining relation \eqref{eq:37}, since $H_{W,\C}$ is dense
 in $\overline H_{w,\C}$.  Any $\xi\in H_W$ is contained in a sublattice
 $L$ of $H_W$ generated by the vertices of a finite (connected)
 subdiagram $\Gamma$ of $\Gamma_W$ (recall \S2.2). Actually, $\Gamma$ is
 a diagram of type either $A_l$ or $D_l$ for some $l\in \Z_{>0}$ and $
 I_W|_L$  gives a root lattice structure of that type on $L$. That is,
 $\{I_W(\gamma_{W,c},\gamma_{w,d})\}_{c,d\in \Gamma\subset C_W}$ is the
 Cartan matrix of type $\Gamma$. In particular, the eigenvalues of $
 \dot{I}_W|_L$ is given by
 $4\sin^2\!\big(\frac{\pi}{2}\frac{m_i}{h}\big)$ ($i\!=\!1,\!\cdots\!,
 l\!=\!\rank(L)$), where $m_i$ are the exponents and $h$ is the Coxeter
 number of the root system of type $\Gamma$ {\small (\cite[ch.V,\S6,n$^o$2]{Bo})}. Since the smallest and the largest exponent of the (finite) root system are $1$ and $h\!-\!1$, respectively, the minimal and the maximal of the eigenvalues are $4\sin^2\!\big(\frac{\pi}{2}\frac{1}{h}\big)$ and $4\cos^2\!\big(\frac{\pi}{2}\frac{1}{h}\big)$, respectively. Since $h\to\infty$ according as $\Gamma$ "exhaust" $\Gamma_w$, we obtain
\begin{eqnarray*}
m &=& \inf_{\Gamma\subset \Gamma_W}  4\sin^2\Big(\frac{\pi}{2}\frac{1}{h}\Big)  =\lim_{h\to\infty}  4\sin^2\Big(\frac{\pi}{2}\frac{1}{h}\Big)  =0.\\
\vspace{-0.5cm}
M &=& \sup_{\Gamma\subset \Gamma_W}  4\cos^2\Big(\frac{\pi}{2}\frac{1}{h}\Big)  = \lim_{h\to\infty}  4\cos^2\Big(\frac{\pi}{2}\frac{1}{h}\Big)  =4.
\vspace{-0.5cm}
\end{eqnarray*}

\vspace{-0.5cm}
\end{proof}

We apply the spectral decomposition theory of bounded symmetric
operators (see \cite[\S107 Theorem]{R-N}) to the operator $\dot{I}_W$.
Let us reformulate the result in [ibid] by adjusting the notation to our setting.

\medskip
\noindent
{\bf Theorem 6.}\ {\it For each type 
{\small $W\in\{\AA,\DD\}$}, there exists a unique
spectral family $\{E_{W,\lambda}\}_{\lambda\in \R}$ (i.e.\ a family of projection operators\footnote
{Here, we mean by a {\it projection operator} an orthogonal projection map
from $\overline H_{W,\C}$ to its closed
subspace such that the real form  $\overline H_{W,\R}$ is mapped into
itself. The fact that $E_{W,\lambda}$ is real, is not explicitly stated
in the literature \cite{R-N}, but follows trivially  from its  construction and from the fact that $\dot{I}_w$ is real.
}
on $\overline H_{W,\C}$ satisfying the following} 
\text{a)}, \text{b)}, \text{c)}: 

\smallskip
\text{a)} {\it For $\lambda\le \mu$, one has $E_{W,\lambda}\le E_{W,\mu}$ ($\underset{\tiny{ def}}{\Leftrightarrow} E_{W,\lambda} E_{W,\mu}=E_{W,\lambda}$).
}

\text{b)} {\it The family is strongly continuous with respect to $\lambda$, i.e.\

\smallskip
\centerline
{
$E_{W,\lambda+0}(:=\underset{\mu\downarrow
0}{\lim}E_{w,\lambda+\mu})=E_{W,\lambda-0}(:=\underset{\mu\uparrow
0}{\lim}E_{w,\lambda-\mu})$,
}
for all $\lambda$ except at $\lambda=2$ for type {\small $W=\DD$}. We have
\begin{equation}
\label{eq:38}
  E_{\DD,2+0}-E_{\DD,2-0}=\eta_{D}
\end{equation}
where $\eta_{D}$ is the projection: 
{\small $\overline H_{\DD,\C}\to\overline{H}_{\DD,2}=\C(\gamma^+_{D,0}-\gamma^-_{D,0})$.}
}

\text{c)} {\it One has $E_{W,\lambda}=0$ for $\lambda\le0$ and
$E_{w,\lambda}=\Id_{\overline{H}_{w,\C}}$ for $\lambda\ge 4$.

\smallskip
\noindent
so that following \eqref{eq:39} holds.
\vspace{-0.1cm}
\begin{equation}
\label{eq:39}
\qquad  (\dot{I}_W)^r=\int_{0}^4\lambda^r dE_{W,\lambda} \qquad (\text{for \ \ } r=0,1,2,\cdots).\quad 
\end{equation}
 where the integral is in the sense of Lebesgue-Stieltjes.  
}\footnote{
Furthermore, \cite[\S107 Theorem]{R-N}, for any complex valued
continuous function $u(\lambda)$ on the interval $[0,4]$, we have an
equality 
$u(\dot{I}_W)= \int_{0}^4 u(\lambda) dE_{W,\lambda}$
between bounded operators, where LHS is defined by a (monotone decreasing) polynomial approximation of $u$ and RHS is given by the norm-limit of the Stieltjes type summation.  Then, for any $\xi,\eta\in\overline H_{W,\C}$, we have
$ \langle u(\dot{I}_W)\xi,\eta\rangle = \int_{0}^4 u(\lambda) d\langle E_{W,\lambda}\xi,\eta\rangle.
$
}

\subsection{Spectra of Coxeter elements}.

Recall that $\lambda\in[0,4]$ in \S4.3 Theorem 6 is the parameter for the
spectra of the intersection form $I_W:=I_W^{(d)}$ for
$d\!\equiv\!0\bmod 4$. What is wonderful, is the fact that this
parameter gives a clue to parametrize the spectra $\theta$ of the Coxeter
elements $Cox_W^{(d)}$ for all  $d\in \Z_{\ge1}$. In order to achieve
this, we re-parametrize 
$\lambda$ by a new parameter $\theta$ and by the relation (which we once used in a proof of {\bf
Assertion 2}.)
\begin{equation}
\label{eq:40}
\qquad \lambda \ = \ 4 \sin^2\Big(\theta\frac{\pi}{2}\Big) 
\quad \ \text{for }\ 0\le\theta\le 1.\!\!\!\!\!\!
\end{equation}
%
Let us introduce a Stieltjes measure on the interval $\theta\in[0,1]$:
\begin{equation}
\label{eq:41}
\xi_{W,\theta}\ :=\  U_\theta \!\cdot\! dE_{W,\lambda} \!\cdot\! U_\theta^{-1}
\end{equation}
for each {\small $W\!\in\!\!\{\AA,\DD\}$}, where 

\noindent
{\rm (i)} $\{E_{W,\lambda}\}_{\lambda\in[0,4]}$ is the spectral
family  in \S{\rm 4.3} {\bf Theorem 6},  

\noindent
(ii) $U_\theta$ $(0\!\le\!\theta\!\le\!1)$ is a family of unitary
operators on 
$\overline{H}_{W,\C}$ given by
\begin{equation}
\label{eq:42}
U_\theta:=\exp\Big(\!\!-\frac{\pi}{2}\sqrt{-1}\theta\Big)\pi_0-\exp\Big(\frac{\pi}{2}\sqrt{-1}\theta\Big)\pi_1,
\end{equation}
where $\pi_i$ $(i\!=\!0,1)$ is the orthogonal projections to
the subspace $\overline H_{W,\C,i}$.

\medskip

\bigskip
\noindent
{\bf Theorem 7.}\
{\it We have the following} \text{a)} {\it and} \text{b)}.

\smallskip
\text{a)} {\it The $\xi_{W,\theta}$ is a Stieltjes measure on $\theta\in[0,1]$,
which is strongly continuous except at $\theta=\frac{1}{2}$ for the type
$\DD$. We have}
\begin{equation}
\label{eq:43}
\xi_{\DD,\frac{1}{2}+0}-\xi_{\DD,\frac{1}{2}-0}=\eta_{D},
\end{equation}
where we recall \eqref{eq:38} for $\eta_D$.

\text{b)} {\it The following two formulae hold:
\begin{equation}
\label{eq:44}
Cox_W^{(d)}\!\cdot\! \xi_{W,\theta} = \exp\Big(2\pi\sqrt{-1}\Big(\theta+\frac{d-2}{2}\Big)\Big)\xi_{W,\theta},
\end{equation}
\vspace{-0.2cm}
and
\vspace{-0.2cm}
\begin{equation}
\label{eq:45}
\int_{\theta=0}^{\theta=1} \xi_{W,\theta} \ =\
\frac{1}{2}\ \dot{I}_W.
\end{equation}
}
\begin{proof} 
a) The first half of the statement is obvious. 
Since
 $\overline{H}_{\DD,2}$ is a subspace of $\overline{H}_{\DD,0}$,
 we have $\pi_0\eta_D\!=\!\eta_D\pi_0\!=\!\eta_D$ and
 $\pi_1\eta_D\!=\!\eta_D\pi_1\!=\!0$. Then, LHS of \eqref{eq:43} is
 given by
{\footnotesize
 $U_{\frac{1}{2}}dE_{\DD,2+0}U_{\frac{1}{2}}^{-1}-U_{\frac{1}{2}}dE_{\DD,2-0}U_{\frac{1}{2}}^{-1}$}

\noindent
{\footnotesize 
$=U_{\frac{1}{2}}(dE_{\DD,2+0}-dE_{\DD,2-0})U_{\frac{1}{2}}^{-1}=U_{\frac{1}{2}}\eta_DU_{\frac{1}{2}}^{-1}=\eta_D$
 \ 
 (c.f.\ \eqref{eq:38}).}

b) 1. Proof of \eqref{eq:44}. 

Consider the infinitesimal form of the formula \eqref{eq:39} for $r\!=\!1$:
\begin{equation}
\label{eq:46}
\dot{I}_W\!\cdot\!  dE_{W,\lambda} \ = \ \lambda dE_{W,\lambda}.
\end{equation}
Substitute the decomposition $dE_{W,\lambda}=\pi_0\!\cdot\! dE_{W,\lambda}+\pi_1\!\cdot\! dE_{W,\lambda}$ in this formula. Then, 
using \eqref{eq:33} ,  the LHS is equal to
\[
\begin{array}{lll}
\dot{I}_W\!\cdot\!  dE_{W,\lambda} &=& (\dot{J}_W+ ^t\!\dot{J}_W)(\pi_0\!\cdot\! dE_{W,\lambda}+\pi_1\!\cdot\! dE_{W,\lambda})\\
&=& 2\ \pi_0\!\cdot\! dE_{W,\lambda}+2\ \pi_1\!\cdot\! dE_{W,\lambda}\\
&&\!\!\! +\ (\dot{J}_W-id)(\pi_0\!\cdot\! dE_{W,\lambda})\ +\ (\dot{J}_W-id)(\pi_1\!\cdot\! dE_{W,\lambda})\\
&&\!\!\!+\ (^t\!\dot{J}_w-id)(\pi_0\!\cdot\! dE_{w,\lambda})\ +\ (^t\!\dot{J}_W-id)(\pi_1\!\cdot\! dE_{w,\lambda}).
\end{array}
\]
On the other hand, recalling \eqref{eq:30} and \eqref{eq:31}, we know that  
\[
\begin{array}{ccc}
(\dot{J}_W\!-\!id)(\pi_1\!\cdot\! dE_{W,\lambda})\!=\!0,\! &(\dot{J}_W\!-\!id)(\pi_0\!\cdot\! dE_{W,\lambda})\in \mathrm{Hom}(\overline{H}_{W,\C},\overline{H}_{W,\C,1}), \\ 
\!(^t\!\dot{J}_w\!-\!id)(\pi_0\!\cdot\! dE_{W,\lambda})\!=\!0,\! &
(^t\!\dot{J}_W\!-\!id)(\pi_1\!\cdot\! dE_{W,\lambda})\in \mathrm{Hom}(\overline{H}_{W,\C},\overline{H}_{W,\C,0}).
\end{array}
\]
Equating this with  $\lambda
 dE_{W,\lambda}\!=\!\lambda\pi_0\!\cdot\! dE_{w,\lambda}\!+\!\lambda\pi_1\!\cdot\! dE_{w,\lambda}$ \eqref{eq:45},
we obtain 
\[
({}^t\!\dot{J}_W\!-\!id)(\pi_1\!\cdot\! dE_{W,\lambda}) \!=\! (\lambda\!-\!2)\pi_0\!\cdot\! dE_{W,\lambda},\ 
(\dot{J}_W\!-\!id)(\pi_0\!\cdot\! dE_{W,\lambda}) \!=\! (\lambda\!-\!2)\pi_1\!\cdot\! dE_{W,\lambda}.
\]
Rewriting these together in matrix expressions, we obtain
\begin{equation}
\label{eq:47}
\dot{J}_W
\left(\!\!
\begin{array}{c}
\pi_0\!\cdot\! dE_{W,\lambda}\\
\pi_1\!\cdot\! dE_{W,\lambda}\\
\end{array}
\!\!\right)
=
\left(\!\!
\begin{array}{cc}
1 & \lambda-2 \\
0 & 1\\
\end{array}
\!\!\right)
\left(\!\!
\begin{array}{c}
\pi_0\!\cdot\! dE_{W,\lambda}\\
\pi_1\!\cdot\! dE_{W,\lambda}\\
\end{array}
\!\!\right).
\end{equation}
\begin{equation}
\label{eq:48}
^t\!\dot{J}_W
\left(\!\!
\begin{array}{c}
\pi_0\!\cdot\! dE_{W,\lambda}\\
\pi_1\!\cdot\! dE_{W,\lambda}\\
\end{array}
\!\!\right)
=
\left(\!\!
\begin{array}{cc}
1 & 0 \\
\lambda-2 & 1\\
\end{array}
\!\!\right)
\left(\!\!
\begin{array}{c}
\pi_0\!\cdot\! dE_{W,\lambda}\\
\pi_1\!\cdot\! dE_{W,\lambda}\\
\end{array}
\!\!\right).
\end{equation}
and, hence, also
\[(^t\!\dot{J}_W)^{-1}
\left(\!\!
\begin{array}{c}
\pi_0\!\cdot\! dE_{W,\lambda}\\
\pi_1\!\cdot\! dE_{W,\lambda}\\
\end{array}
\!\!\right)
=
\left(\!\!
\begin{array}{cc}
1 & 0 \\
2-\lambda & 1\\
\end{array}
\!\!\right)
\left(\!\!
\begin{array}{c}
\pi_0\!\cdot\! dE_{W,\lambda}\\
\pi_1\!\cdot\! dE_{W,\lambda}\\
\end{array}
\!\!\right).
\]
Thus, combining these with the expression \eqref{eq:32}, we obtain
\begin{equation}
\label{eq:49}
Cox_W^{(d)}\!\!
\left(\!\!
\begin{array}{c}
\!\pi_0\!\cdot\! dE_{W,\lambda}\!\\
\!\pi_1\!\cdot\! dE_{W,\lambda}\!\\
\end{array}
\!\!\right)
\!=\!
(-1)^{d\!-\!1}\!\!\left(\!\!
\begin{array}{cc}
\!1\!&\! \lambda-2\!\!\\
\!\!2\!-\!\lambda \!&\! 1\!-\!(\lambda\!-\!2)^2\!\!\\
\end{array}
\!\!\right)\!\!
\left(\!\!
\begin{array}{c}
\pi_0\!\cdot\! dE_{W,\lambda}\\
\pi_1\!\cdot\! dE_{W,\lambda}\\
\end{array}
\!\!\right).\!\!
\end{equation}
Substitute $\lambda$ in the RHS matrix by the expression \eqref{eq:40} :
{\footnotesize
\[\begin{array}{llll}
(-1)^{d\!-\!1} \left(\!\!\!
 \begin{array}{cc}
 1& \lambda-2\\
 2-\lambda& 1-(\lambda-2)^2\\
 \end{array}
\!\!\!
 \right) 
 = 
(-1)^{d\!-\!1}\left(\!\!\!
\begin{array}{cc}
1& -2\cos(\pi\theta)\\
2\cos(\pi\theta) & \sin^2(\pi\theta)-3\cos^2(\pi\theta)
\end{array}
\!\!\!\right).
\end{array}\]
}
\noindent
We see that the matrix is semi-simple for any $\theta$.  The eigenvalues  are
\vspace{-0.2cm}
\[
 \exp\!\Big(\!\pm2\pi\sqrt{-1}\Big(\ \theta+\frac{d-2}{2}\Big)\Big),
\vspace{-0.2cm}
\]
and associated row eigenvectors (independent of $n$) are  
\begin{equation*}
\vspace{-0.2cm}
\left(\exp\Big(\mp\frac{\pi}{2}\sqrt{-1}\theta\Big), -\exp\Big(\pm\frac{\pi}{2}\sqrt{-1}\theta\Big)\right).
\end{equation*}
Therefore, by introducing the unitary operators
\begin{equation}
\label{eq:50}
U_{\pm\theta}:=
\exp\Big(\mp\frac{\pi}{2}\sqrt{-1}\theta\Big)\pi_0-\exp\Big(\pm\frac{\pi}{2}\sqrt{-1}\theta\Big)\pi_1
\end{equation}
satisfying relations: $^t\!U_{\pm\theta}=U_{\pm\theta}=\overline{U_{\mp\theta}}$
and $U_{\pm\theta}\!\cdot\!
 U_{\mp\theta}=\id_{\overline{H}_{W,\C}}$, we
 introduce a Stieltjes measure on $[0,4]:=\{\lambda\!\in\!\R\mid
 0\!\le\!\lambda\!\le\!4\}\simeq [0,1]:=\{\theta\!\in\!\R\mid 0\!\le\!\theta\!\le\!1\}$:
\begin{equation}
\label{eq:51}
\xi_{W,\theta}^{\pm} := U_{\pm\theta} \!\cdot\! dE_{W,\lambda}\!\cdot\! U_{\mp\theta}.
\end{equation}
Then, from \eqref{eq:49}, we obtain
\begin{equation}
\label{eq:52}
Cox_W^{(d)}\!\cdot\! \xi_{W,\theta}^{\pm} =
\exp\Big(\pm2\pi\sqrt{-1}\Big(\theta+\frac{d-2}{2}\Big)\Big)\xi_{W,\theta}^{\pm}.
\end{equation}
Putting $\xi_{W,\theta}:=\xi_{W,\theta}^+$, we obtain
\eqref{eq:43}.

\smallskip
b) 2. Proof of \eqref{eq:45}.

Using \eqref{eq:41} and \eqref{eq:42}, we decompose $\xi_{W,\theta}$ into 4 pieces:

\smallskip
\noindent
{\small 
\[
 \pi_0\!\cdot\! dE_{W,\theta}\!\cdot\! \pi_0 \!+\! 
\pi_1\!\cdot\! dE_{W,\theta}\!\cdot\! \pi_1 \!-\! 
\exp(\pi\sqrt{\!-\!1}\theta)\pi_1\!\cdot\! dE_{W,\theta}\cdot \pi_0 \!-\! 
 \exp(\!-\!\pi\sqrt{\!-\!1}\theta)\pi_0\!\cdot\! dE_{W,\theta}\!\cdot\!
 \pi_1.
\]
}
\noindent
The first two terms are integrated easily by
{\small
\begin{eqnarray*}
\int_{\theta=0}^{\theta=1}
\pi_0\!\cdot\! dE_{W,\theta}\!\cdot\! \pi_0 =
\pi_0\!\cdot\!\left(\int_{\theta=0}^{\theta=1}
 dE_{W,\theta}\right)\!\cdot\! \pi_0 =
\pi_0\cdot\id_{\overline{H}_{W,\C}}\cdot\pi_0=\pi_0,\\
\int_{\theta=0}^{\theta=1}
\pi_1\!\cdot\! dE_{W,\theta}\!\cdot\! \pi_1 =
\pi_1\!\cdot\!\left(\int_{\theta=0}^{\theta=1}
 dE_{W,\theta}\right)\!\cdot\! \pi_1 =
\pi_1\cdot\id_{\overline{H}_{W,\C}}\cdot\pi_1=\pi_1.
\end{eqnarray*}
}
The third and fourth terms are integrated by the use of Footnote 8. 

First, we introduce bounded nilpotent operators
 $\dot{K}_W\!:\!\overline{H}_{W,0,\C}\!\to\! \overline{H}_{W,1,\C}$ and 
 $^t\!\dot{K}_W\!:\!\overline{H}_{W,1,\C}\!\to\! \overline{H}_{W,0,\C}$, by 
$\dot{K}_W\!:=\!\id_{\overline{H}_{W,\C}}\!-\!\dot{J}_W$ and 
$^t\!\dot{K}_W\!:=\!\id_{\overline{H}_{W,\C}}\!-\!^t\!\dot{J}_W$ so that we have
$\dot{K}_W^2\!= ^t\!\!\dot{K}_W^2\!=\!0$
 and  $\dot{I}_{W}\!=\!2\ \id_{\overline{H}_{W,\C}}\!\!-\!\dot{K}_W\!-^t\!\!\dot{K}_W$.
 Then,\footnote
{In the present paper, $\sqrt{X}$ takes the positive branch for a
 positive object $X$.}
\[\begin{array}{llll}
& \int_{\theta=0}^{\theta=1}
\exp(\pi\sqrt{\!-\!1}\theta)\pi_1\!\cdot\! dE_{W,\theta}\cdot \pi_0 \\
 =&
  \pi_1\left[\int_{\theta=0}^{\theta=1}\left(1\!-\!2\sin^2(\frac{\pi}{2}\theta)\!+\!\sqrt{-1}
\ 2\sqrt{1\!-\!\sin^2(\frac{\pi}{2}\theta)}\sin(\frac{\pi}{2}\theta)\right) dE_{W,\lambda}\right]\pi_0
\\
 =&
  \pi_1\left[\int_{\theta=0}^{\theta=1}\left(1-\frac{\lambda}{2}+\frac{\sqrt{-1}}{2}\sqrt{(4-\lambda)\lambda}\right) dE_{W,\lambda}\right]\pi_0
\\
 =&
  \pi_1\left[\id_{\overline{H}_{w,\C}}-\frac{\dot{I}_W}{2}+\frac{\sqrt{-1}}{2}\sqrt{(4\
	\id_{\overline{H}_{w,\C}}-\dot{I}_W)\dot{I}_W} \right]\pi_0
\\
\end{array}
\]
After sandwiching by $\pi_1$ and $\pi_0$, the first and the second
 terms turn out to be $\pi_1\cdot\id_{\overline{H}}\cdot\pi_0=0$ and
 $\pi_1\cdot\frac{\dot{I}_W}{2}\cdot\pi_0=-\frac{\dot{K}_W}{2}$, respectively.
The third term turns out to be zero, since the operator 
{\small
\[ \begin{array}{lll}
\sqrt{(4\ \id_{\overline{H}_{W,\C}}\!-\!\dot{I}_W)\dot{I}_W}
 \!&=\sqrt{(2\ \id_{\overline{H}_{W,\C}}\!\!+\!\dot{K}_W\!+\!\ ^t\!\dot{K}_W)(2\
  \id_{\overline{H}_{W,\C}}\!\!-\!\dot{K}_W\!-\! ^t\!\dot{K}_W)} \\
&= \sqrt{4\ \id_{\overline{H}_{W,\C}}-\dot{K}_W\cdot\ ^t\!\dot{K}_W-\ ^t\!\dot{K}_W\cdot \dot{K}_W} 
\end{array}
\]
}
preserves the decomposition \eqref{eq:29} so that it does not have the
 ``cross'' term sandwiched by $\pi_1$ and $\pi_0$. 
Thus, we get 
{\small
\[
\int_{\theta=0}^{\theta=1}
\exp(\pi\sqrt{\!-\!1}\theta)\pi_1\!\cdot\! dE_{w,\lambda}\cdot \pi_0 
=\frac{\dot{K}_W}{2}. 
\]
}
\vspace{-0.2cm}
Similarly, we obtain also
{\small
\[
 \int_{\theta=0}^{\theta=1}
\!\!\exp(-\!\pi\sqrt{\!-\!1}\theta)\pi_0\!\cdot\! dE_{W,\lambda}\cdot \pi_1 
=\frac{^t\!\dot{K}_W}{2}. 
\]
}
These altogether show the formula \eqref{eq:45}
\end{proof}

\noindent
{\bf Corollary.} 
{\it Let
$\varphi(\theta)=\sum_{m\in\Z}a_m\exp(2\pi\sqrt{-1}m(\theta+\frac{d-2}{2}))$
be an absolutely convergent Fourier expansion of a complex valued continuous function on
the interval $\theta\in [0,1]$. Then, we have
{\small
\begin{equation}
\label{eq:53}
2 \int_{\theta=0}^{\theta=1}\varphi(\theta)\cdot\xi_{W,\theta} =
\sum_{m\in\Z}a_m (Cox_W^{(d)})^m \cdot \dot{I}_W.
\end{equation}
}
}

\medskip
\noindent
{\bf Remark.}  Due to the integral formula \eqref{eq:45}, we get a factor $\dot{I}_W$ at the end of the
formula \eqref{eq:53}. Since $\dot{I}_W$ is not invertible by a bounded
operator (recall the comment at the end of \S4.2), the meaning of this factor is unclear. It would be desirable
to ask: 

\smallskip
\noindent
 {\bf Question:} Can $\xi_{W,\theta}$ be divisible by $\dot{I}_W$ from
 the right (c.f.\ \S4.3 Corollary to {\bf Assertion 2})?

\medskip
\bigskip
\noindent
{\it Acknowledgements.} The author is grateful to Hikosaburo Komatsu and
Serge Richard for discussions to clarify the spectral analysis in \S4.


\begin{thebibliography}{[Sa:20]}
\bibitem[Bo]{Bo}
    N. Bourbaki: 
    groupes et algr\'ebres de Lie, Chapitres 4,5, et 6. El\'ements de Math\'ematique XXXIV. Paris: Hermann 1968.

\bibitem[Boa]{Boa}
R.\ P.\ Boas:
Entire Functions, Academic Press Inc., 1954 New York.


\bibitem[R-N]{R-N}
F.\ Riesz and B.\ Nagy: Lecons d'analyse fonctionnelle, Academie des science de Hungrie, Akademiei Nyomda (1955).

\bibitem[Sa1]{Sa1}
K.\ Saito: Period mapping associated to primitive forms, Publ. RIMS {\bf
	     19} (1983), no.3, p1231-1264, 32G11(32B30).

\bibitem[Sa2]{Sa2}
K.\ Saito: Polyhedra dual to the Weyl Chamber decomposition: A Precis. Publ.\ RIMS, Kyoto Univ., {\bf 40} (2004), no.4
	     p1337-1384, 14P10(20F55)

\bibitem[Sa3]{Sa3}
K.\ Saito: Polyhedra Dual to Weyl Chamber Decompsotion, Preprint, 2004 Augst.
%
	     no.2, p645-668, 52B05(05C05 20F36).

\bibitem[Sa4]{Sa4}
K.\ Saito: Einfach elliptische Singularit\"aten, Inventiones Math. {\bf
	    23 } (1974), 289-325.
\end{thebibliography}
\end{document}